\documentclass[preprint,12pt]{elsarticle}

\usepackage{amssymb}
\usepackage{amsmath}

\usepackage{amsthm}
\newtheorem{theorem}{Theorem}
\newtheorem{definition}{Definition}
\newtheorem{prop}{Proposition}

\newtheorem{conj}{Conjecture}
\newtheorem{corollary}{Corollary}
\newtheorem{lemma}{Lemma}
\newtheorem*{prop: binomial difference bound}{Proposition \ref{binomial difference bound}}
\newtheorem*{thm: box fragmentation conj answer arbitrary dim}{Theorem \ref{box fragmentation conj answer arbitrary dim}}
\newtheorem*{conj: box fragmentation conjecture}{Conjecture \ref{box fragmentation conjecture}}

\usepackage{hyperref}
\usepackage{bbm}
\usepackage[margin=1in]{geometry}

\newcommand{\R}{\ensuremath{\mathbb{R}}}
\newcommand{\Q}{\mathbb{Q}}
\newcommand{\Z}{\ensuremath{\mathbb{Z}}}

\journal{Journal of Number Theory}

\begin{document}
\begin{frontmatter}


\title{Benford Behavior Resulting From Stick and Box Fragmentation Processes}

\author[inst1]{Bruce Fang\corref{cor1}}
\cortext[cor1]{Corresponding author.}
\ead{fangbaojun2002@gmail.com}

\author[inst1]{Steven J. Miller}
\ead{sjm1@williams.edu or steven.miller.mc96@aya.yale.edu}

\affiliation[inst1]{organization={Department of Mathemtics},
             addressline={Williams College},
             city={Williamstown},
             postcode={01267},
             state={MA},           country={USA}}

%\author{} %% Author name

%\affiliation{organization={},%Department and Organization

%% Abstract
\begin{abstract}
Benford’s law is the statement that in many real-world data set, the probability of having digit $d$ in base $B$ as the first digit is $\log_{B}\left((d+1)/d\right)$ for all $1\leq d\leq B$. We sometimes refer to this as weak Benford behavior, and we say that a data set satisfies strong Benford behavior in base $B$ if the probability of having significand at most $s$ is $\log_{B}(s)$ for all $s\in[1,B)$. We examine Benford behaviors in two different probabilistic model: stick and box fragmentation models. Building on the joint work of Becker et al. \cite{dependent} on the single proportion stick fragmentation model, we employ combinatorial identities on multinomial coefficients to reduce the multi-proportion stick fragmentation model to the single proportion model. We then provide a necessary and sufficient condition for the lengths of the stick fragments to converge to strong Benford behavior along with a quantification of the discrepancy from uniform distribution on $[0,1]$ in terms of irrationality exponent. Then we answer a conjecture posed by Betti et al. \cite{box fragmentation} on the high dimensional box fragmentation model. Using tools from Fourier analysis and order statistics, we prove that under some conditions, faces of any arbitrary dimension of the box have total volume converging to strong Benford's behavior.
\end{abstract}

%% Keywords

\end{frontmatter}

%\linenumbers


\tableofcontents

\section{Introduction}
At the end of the 19\textsuperscript{th} century, astronomer and mathematician Simon Newcomb \cite{Ne} noticed an unusual pattern in the logarithmic tables he used at work. The early pages of the books were far more worn than the later ones, suggesting that numbers starting with smaller digits were consulted more often. Specifically, Newcomb found that $1$ appeared as the leading digit for about $30\%$ of the time, $2$ for about $17\%$, with the frequency decreasing for larger digits. Although he formulated a mathematical law for this curious phenomenon, his discovery initially went largely overlooked.

It took over 50 years after Newcomb's discovery for physicist Frank Benford to make the same observation as Newcomb. He formulated this law as follows. Although the law was originally only stated for base 10, we state the generalized form for any base $B\geq 2$.

\begin{definition}\cite[Page 554]{Ben} A data set is said to satisfy \textbf{Benford's law for the leading digit} in base $B$ if the frequency $F_d$ of leading digit $d$ is
\begin{align}
F_d \ = \ \log_{B}\left(\frac{d+1}{d}\right).
\end{align}
\end{definition}

Nowadays, Benford's Law is used in detecting many different forms of fraud, and its prevalence in the world fascinates not only mathematicians, but many other scientists as well (to learn more about Benford's law and its many applications, we recommend~\cite{BeHi, Mil1, Ni} to name a few). 

In 1986, Lemons \cite{Lemons} proposed applying Benford's law to the analysis of the partitioning of a conserved quantity. Motivated in part by potential applications to nuclear fragmentation, both mathematicians and physicists have since investigated the Benford behavior of a range of fragmentation models. One such model is \textit{stick fragmentation}.
In a 1-dimensional stick fragmentation model, we begin with a stick of length $L$. Draw $p_1$ from a probability distribution on $(0,1)$. This divides the stick into two sub-sticks of lengths $p_1L$ and $(1-p_1)L$. For each resulting sub-stick, draw two probabilities $p_2$ and $p_3$ from the same distribution and fragment them accordingly. This procedure is repeated for a total of $N$ stages, with every sub-stick produced at a given stage being fragmented at the next stage using newly drawn probabilities. After $N$ stages, the process yields $2^N$ sticks. A central question is whether the distribution of stick lengths produced by this process converges to Benford's behavior.

\subsection{Previous Work on Fragmentation}

A key definition in formulating a more general version of Benford behavior is the notion of the significand of a real number, i.e., its leading digits in scientific notation.

\begin{definition}
For any $x>0$, we can uniquely write
\begin{align}
x \ = \ S_{B}(x)\cdot B^{k_B(x)},
\end{align}
where $S_B(x)\in [1,B)$ and $k_{B}(x)\in \Z$. Equivalently, $k_{B}(x)=\lfloor \log_B(x)\rfloor$ and $S_{B}(x)=x/B^{k_B(x)}$. We call $S_B(x)$ the \textbf{significand} of $x$. 
\end{definition}

We now define a more general version of Benford behavior that captures processes that only display Benford behaviors in the limit.

\begin{definition}
A sequence of random variables $\{X^{(N)}\}_{N=1}^\infty$ is said to converge to \textbf{strong Benford behavior} in base $B$ if 
\begin{equation}
\lim_{N\rightarrow\infty}\mathbb{P}(S_B(X^{(N)}) \leq s) \ = \ \log_B(s)
\end{equation}
for all $s \in [1,B)$..
\end{definition}

An equivalent formulation to the above is the \textbf{Uniform Distribution Characterization}, which is especially suited for investigation of products of random variables; see \cite{Dia} for a proof.
\begin{prop}[Uniform Distribution Characterization]\cite{Dia}\label{uniform distribution characterization prop}
A sequence of random variables $\{X^{(N)}\}_{N=1}^\infty$ converges to strong Benford behavior in base $B$ if and only if
\begin{align}\label{uniform distribution characterization condition}
\lim_{N\rightarrow\infty}\mathbb{P}\left(\log_B(X^{(N)})\textup{ 
 mod  } 1\leq t\right) \ = \ t,
\end{align}
for all $t\in [0,1)$. If \eqref{uniform distribution characterization condition} is satisfied, then we say that $\{\log_B(X^{(N)})\}_{N=1}^\infty$ \textbf{converges to being equidistributed mod 1}.
\end{prop}

We are now ready to review some previous results on stick fragmentation models and introduce the two models of our paper. The first model is an extension of the model introduced by Becker et al. called the fixed single proportion stick fragmentation model \cite{dependent}. Their main result on this model is a necessary and sufficient condition for the distribution of stick lengths to converge to strong Benford behavior as well as a quantification of the error term in terms of irrationality exponent, which we define now.

\begin{definition}\cite{power saving}
Suppose that $x$ is a real number. The \textbf{irrationality exponent} $\mu_x$ of $x$ is the supremum of the set of $\mu$ such that $0<|x-p/q|<1/q^\mu$ is satisfied by an infinite number of coprime integer pairs $(p,q)$ with $q>0$. If such a set does not exist, then we say $x$ has irrationality exponent $\infty$.
\end{definition}

\begin{definition}\cite{power saving}
For a finite sequence $\{x_i\}_{i=1}^N$, denote by $A([a,b),N)$ the number of $x_i$'s such that $x_i \textup{ mod 1} \in [a,b)$ for $0\leq a<b\leq 1$. Then we call the number
\begin{align}
D_N \ := \ \sup_{0\leq a< b\leq 1}\left|\frac{A([a,b),N)}{N}-(b-a)\right|
\end{align}
the \textbf{discrepancy} of the sequence. It measures how close a sequence is from being uniform mod 1.
\end{definition}

Now we are ready to state a main theorem of \cite{dependent}.

\begin{theorem}\label{single proportion thm}\cite{dependent}
Fix a proportion $p \in (0,1)$ and a stick of length $L$. At Stage 1, cut the stick at proportion $p$ to create two sub-sticks of lengths $Lp$ and $L(1-p)$. At  Stage 2, cut each of these two sub-sticks into two smaller sub-sticks at the same proportion $p$. Repeat this procedure for a total of $N$ stages, generating $2^N$ sticks with $N+1$ distinct lengths (assuming $p \neq 1 /2 $) given by
\begin{align} 
x_1 &\ = \ Lp^N \nonumber\\
x_2 & \ = \ Lp^{N-1}(1-p)\nonumber\\
x_3 & \ = \ Lp^{N-2}(1-p)^2\nonumber\\
&\quad \vdots \nonumber\\
x_{N} &\ = \ Lp(1-p)^{N-1}\nonumber\\
x_{N+1} &\ = \   L(1-p)^N,
\end{align}
where the frequency of $x_n$ is $\binom{N}{n}/2^N$. Consider the sequence of random lengths $\{X^{(N)}\}_{N=1}^\infty$, where $X^{(N)}$ is the length of a stick picked uniformly at random from the $2^N$ sticks obtained at stage $N$.  Choose $y$ so that $B^y = (1-p)/p$, which is the ratio of adjacent lengths (i.e., $x_{i+1}/x_i$). Then $\{X^{(N)}\}_{N=1}^\infty$ converges to strong Benford behavior base in $B$ if and only if $y\not\in\Q$. If $y$ has a finite irrationality exponent, then for sufficiently large $N$, the discrepancy of the sequence of the $2^N$ stick lengths (not distinct) can be quantified in terms of that exponent, and there is a power savings.
\end{theorem}

Theorem \ref{single proportion thm} suggests the possibility of generalizing the fixed single proportion stick fragmentation model to a fixed multi-proportion stick fragmentation model, where at each stage, instead of cutting a stick into 2 sub-sticks according to a fixed proportion, we cut a stick into an arbitrary but fixed $m$ number of sub-sticks according to $m-1$ fixed proportions. Again, we are interested in the conditions under which the distribution of stick lengths produced in this model converges to strong Benford's behavior. A main result of this paper provides a definitive answer to this question, as follows:

\begin{theorem}\label{multiproportion thm}
For a fixed integer $m>2$, choose $p_1, p_2, \dots, p_{m-1}\in (0,1)$ such that $p_1+p_2+\dots+p_{m-1}<1$. Set $p_m:=1-(p_1+p_2+\dots+p_{m-1})$. Now, consider a stick of length $L$. At each stage, we cut every remaining stick into $m$ sub-sticks according to proportions $p_1,p_2,\dots,p_{m-1}$. After stage $N$, we have $m^N$ sticks in total, of lengths
\begin{align}
A^{(N)}_{k_1,k_2,\dots,k_m} \ = \ Lp_1^{k_1}p_2^{k_2}\cdots p_m^{k_m},
\end{align}
for all $0\leq k_1,k_2,\dots, k_m\leq N$ such that $k_1+k_2+\cdots+k_m=N$. Consider the sequence of random lengths $\{X^{(N)}\}_{N=1}^\infty$, where $X^{(N)}$ is the length of a stick picked uniformly at random from the $m^N$ sticks obtained at stage $N$. Let $y_i=\log_{B}(p_i/p_{i+1})$ for all $1\leq i\leq m-1$. Then $\{X^{(N)}\}_{N=1}^\infty$ converges to strong Benford behavior base in $B$ if and only if $y_i\not\in\Q$ for some $1\leq i\leq m-1$. Let $\kappa_0$ be the least irrationality exponent among all the irrational $y_i$'s. Then for sufficiently large $N$, the discrepancy of the sequence of the $m^N$ stick lengths (not distinct) is $O(N^{\delta(-1/(\kappa_0-1)+\epsilon')})$ for some $\delta>0$ and for every $\epsilon'>0$.
\end{theorem}

In \cite{dependent}, Becker et al. also looked into other variants of the stick fragmentation model, in particular, the Unrestricted 1-Dimension Decomposition Model, where the proportion cuts are independently drawn from the same unspecified probability density $f$. They were able to establish a sufficient condition in terms of the Mellin transforms of the probability density function $f$ for strong Benford behavior to emerge from the stick length distribution. This direction was generalized by Durmi\'{c} and Miller to a higher dimensional model called the box fragmentation model in \cite{DM}, and strong Benford behavior was established for the volume of a fragmented box of arbitrary dimension $m$ under a similar condition on the Mellin transforms of $f$. In \cite{box fragmentation}, Betti et al. further generalized the work of Durmi\'{c} and Miller on box fragmentation model and considered the volume of the lower $d$-dimensional faces of a fragmented box of arbitrary dimension $m$, where $d\leq m$. To state their results, we start with some definitions.
\begin{definition}
A set $\mathfrak{B}\subset\R^m$ is said to be an $m$-\textbf{dimensional box} if it is a set of the form $[a_1,b_1]\times\cdots\times [a_m,b_m]\subset\R^m$, where $a_i<b_i$
 are real numbers.    
\end{definition}
\begin{definition}
A \textbf{linear-fragmentation process} is a sequence of random variables $\mathfrak{B}_0,\mathfrak{B}_1,$ $\mathfrak{B}_2,\dots$ such that the following holds.
\begin{enumerate}
\item The random variables $\mathfrak{B}_i$ are $m$-dimensional boxes.
\item The random variables $\mathfrak{B}_i$ form a descending chain $\mathfrak{B}_0\supset \mathfrak{B}_1\supset \mathfrak{B}_2\supset \dots$.
\item The distribution $\mathfrak{B}_{n+1}$ conditioned on $\mathfrak{B}_n$ is some fixed distribution of independent proportion cuts $P_1,\dots,P_m$ along each Cartesian axis. These $P_i$ are fixed over all $n\geq 0$.
\item The proportion cuts $P_i$ are continuous random variables.
\item $\mathbb{E}[\log_B P_i]=\mu_P\in \R$ and $\textup{Var}[\log_B P_i]=\sigma_P^2>0$ for all $1\leq i\leq m$.
\end{enumerate}
\end{definition}

\begin{definition}
Given an $m$-dimensional box $\mathfrak{B}$ and a positive integer $d\leq m$, we say the $d$-\textbf{volume} of $\mathfrak{B}=\prod_{i}[a_i,b_i]$ is the sum of the $d$-dimensional volumes of the $d$-dimensional faces of $\mathfrak{B}$:
\begin{align}
\textup{Vol}_d(\mathfrak{B}) \ := \ 2^{m-d} \sum_{|I|=d}\prod_{i\in I}(b_i-a_i),
\end{align}
where we are summing over all subsets $I\subset \{1,\dots, m\}$ with cardinality $d$.
\end{definition}
\cite{box fragmentation} established a sufficient condition for strong Benford behavior to emerge (see their Theorem 1.9), which involves the volume of the largest $d$-dimensional face.
\begin{theorem}[Maximum Criterion]\cite{box fragmentation}\label{maximum criterion}
Let $\mathfrak{B}:=\mathfrak{B}_0$ be a fixed $m$-dimensional box and $\mathfrak{B}_0\supset \mathfrak{B}_1\supset \cdots$ be a linear-fragmentation process with proportion cuts $P_i$ having probability density functions $f_i:(0,1)\rightarrow (0,\infty)$. Let
\begin{align}
V_d^{(N)} \ := \ \textup{Vol}_d(\mathfrak{B}_N)
\end{align}
be the sequence of volumes obtained from this process and $\mathfrak{m}_d^{(N)}$ denote the maximum product of $d$ sides of $\mathfrak{B}_N$. If $\mathfrak{m}_d^{(N)}$ converges to strong Benford behavior in base $B$ as $N\rightarrow\infty$, then $V_d^{(N)}$ also converges to strong Benford behavior in base $B$ as $N\rightarrow\infty$.
\end{theorem}
In the same paper, the Maximum Criterion was verified for $d=1$ in a special case.
\begin{theorem}
Let $P_i^{(j)}$ be i.i.d. with $\log_B P_i^{(j)}$'s having uniform distributions. When $d=1$, the maximum side length of each box
\begin{align}
\mathfrak{m}_1^{(N)} \ = \ \textup{max}_{1\leq i\leq m}P_i^{(1)}\cdots P_i^{(N)}
\end{align}
converges to strong Benford behavior in base $B$ as $N\rightarrow\infty$.
\end{theorem}
It was conjectured in the same paper (see Conjecture 4.1) that for any linear fragmentation process, $m_d^{(N)}$ converges to strong Benford behavior.
\begin{conj}\label{box fragmentation conjecture}
Every linear fragmentation process satisfies the Maximum Criterion in all dimensions $1\leq d\leq m$.
\end{conj}
We prove the conjecture for any $m\geq 2$ and $d\leq m$ under some conditions on the linear fragmentation process, which we shall explain more in Section \ref{higher dimension}.
\begin{theorem}\label{box fragmentation conj answer arbitrary dim}
Suppose that $\mathfrak{B}_0\supset \mathfrak{B}_1\supset \mathfrak{B}_2\supset \cdots$ is a linear fragmentation process on an $m$-dimensional box for an arbitrary $m\geq 2$, such that the proportion cut $P_i^{(j)}$ at every stage satisfies the following.
\begin{enumerate}
\item Condition 1: $P_i^{(j)}$'s are i.i.d. for all $1\leq i\leq m$ and $1\leq j\leq N$. That is, $P_i^{(j)}$'s are all independent and there exists a random variable $P$ such that $P_i^{(j)}\sim P$ for all $i,j$;
\item Condition 2: Let the mean and the variance of $\log_B(P)$ be $\mu_P$ and $\sigma_P$. Then $\log_B(P)$ has finite centered moments, i.e., $\mathbb{E}[(\log_B(P)-\mu_P)^k]<\infty$ for all $k\geq 1$ and the characteristic function of $(\log_B(P)-\mu_P)/\sigma_P$ admits Taylor expansion. Moreover, the variance is nonzero;
\item Condition 3: Let $f(x)$ be the probability density function of $\log_B(P)$, then we have $\textup{ess sup}_x f(x)<\infty$;
\item Condition 4: $f(x)$ is differentiable with $\textup{sup}_x |f'(x)|<\infty$.
\end{enumerate}
Then this linear fragmentation process satisfies the condition of the Maximum Criterion for any $d\leq m$, i.e., $\mathfrak{m}_d^{(N)}$ converges to strong Benford behavior in base $B$ as $N\rightarrow\infty$ for any $d\leq m$, thus by Theorem \ref{maximum criterion}, $V_d^{(N)}$ converges to strong Benford behavior in base $B$ as $N\rightarrow\infty$ for any $d\leq m$.
\end{theorem}

\section{Proof of Theorem \ref{multiproportion thm}}

\subsection{Preliminary}\label{prelim subsection} In \cite{FILM}, Fang, Irons, Lippelman and Miller proved that if $y_i\in\Q$ for all $1\leq i\leq m-1$, the distribution of the logarithms mod 1 of the stick lengths converges to a discrete distribution instead of a uniform distribution on $[0,1]$. Hence, having $y_i\not\in\Q$ for some $1\leq i\leq m-1$ is a necessary condition for the distribution of the stick lengths to follow strong Benford behavior.

For the remainder of this section, we show that having $y_i\not\in\Q$ for some $1\leq i\leq m-1$ is also a sufficient condition for the distribution of the stick lengths to converge to strong Benford's behavior. After reordering the $p_i^{k_i}$'s and thus the $y_i$'s in the factorization of the stick length $A^{(N)}_{k_1,k_2,\dots, k_m}:=p_1^{k_1}p_2^{k_2}\cdots p_m^{k_m}$, it suffices to show that having the first exponent $y_1\not\in\Q$ is a sufficient condition for the distribution of the stick lengths to converge to strong Benford's behavior.

We know that after stage $N$, the number of sticks of length $A^{(N)}_{k_1,k_2,\dots,k_m}$ is
\begin{align}
\binom{k_1+k_2+\cdots+k_m}{k_1,k_2,\dots,k_m} \ = \ \binom{N}{k_1,k_2,\dots,k_m},
\end{align}
where the multinomial coefficients are defined in Definition \ref{multinomial coefficient defn}. From the proof of Theorem 1 of \cite{FILM}, we know that the stick length $A^{(N)}_{k_1,k_2,\dots,k_m}$ has the factorization
\begin{align}
A^{(N)}_{k_1,k_2,\dots,k_m} \ = \ p_1^{k_1}p_2^{k_2}\cdots p_m^{k_m} \ = \ \left(\frac{p_1}{p_2}\right)^{k_1}\left(\frac{p_2}{p_3}\right)^{k_1+k_2}\cdots\left(\frac{p_{m-1}}{p_m}\right)^{\sum_{j=1}^{m-1}k_j}(p_m)^N,
\end{align}
which motivates us to define the exponents $y_i=\log_B(p_i/p_{i+1})$ for $1\leq i\leq m-1$. Now, fix an interval $(a,b)\subset (0,1)$. Let
\begin{align}
\chi_{(a,b)}(x) \ : = \ \mathbbm{1}\Big(\log_B(x)\textup{ mod 1}\in (a,b)\Big)
\end{align}
be the indicator function for $\log_B(x)\textup{ mod 1}\in (a,b)$. Then after stage $N$, the probability that the logarithm mod 1 of a stick length is in $(a,b)$ equals
\begin{align}\label{formula for probability of log mod 1 of stick length between a and b}
F_N(a,b) \ := \ \frac{1}{m^N}\sum_{\substack{0\leq k_1,k_2\dots,k_m \leq N \\ k_1+k_2+\cdots+k_m=N}}\binom{N}{k_1,k_2,\dots,k_m}\chi_{(a,b)}\left(\prod_{i=1}^m p_i^{k_i}\right).
\end{align}
Our goal is to prove that as $N\rightarrow\infty$,
\begin{align}
F(a,b) \ := \ \lim_{N\rightarrow\infty}F_N(a,b) \ = \ b-a,
\end{align}
which by the Uniform Distribution Characterization implies strong Benford behavior.

\subsection{Combinatorial identities} The key steps to our proof are two combinatorial identities related to multinomial coefficients, which we present in this subsection. We first review the definition of multinomial coefficients.
\begin{definition}\label{multinomial coefficient defn}
Suppose that $n, k_1, k_2, \dots, k_m\geq 0$ and $k_1+\cdots+k_m=n$, then we define the multinomial coefficients by
\begin{align}
\binom{n}{k_1, k_2,\dots, k_m} \ := \ \frac{n!}{k_1!k_2!\cdots k_m!}.
\end{align}
\end{definition}
The following identity allows us to express multinomial coefficients as the product of binomial coefficients. As we shall see, this forms a key ingredient to our proof, as it helps reduce our multi-proportion problem to the original single-proportion problem in \cite{dependent}. %\textcolor{red}{Describe the multinomial coefficient identity}
\begin{lemma}\label{multinomial identity I}\cite{Stan}
For $m\geq 1$,
\begin{align}
\binom{k_1+k_2+\cdots+k_m}{k_1,k_2,\dots, k_m} \ = \ \prod_{i=1}^m\binom{k_1+\cdots+k_i}{k_i}.
\end{align}    
\end{lemma}
The proof for Lemma \ref{multinomial identity I} follows a straightforward inductive argument based on the expansion of the multinomial and binomial coefficients. Below is another identity that comes in handy when evaluating sums of multinomial coefficients.
\begin{lemma}\label{multinomial identity II}\cite{Stan}
For $N\in \Z_{\geq 0}$ and $m\in \Z_{>0}$,
\begin{align}
m^N &\ = \ \sum_{\substack{0\leq k_1,k_2,\dots,k_m\leq N \\ k_1+k_2+\cdots+k_m=N}}\binom{N}{k_1,k_2,\dots, k_m} \ = \ \sum_{\substack{0\leq k_1,\dots, k_{m-1}\leq N \\ k_1+k_2+\cdots+k_{m-1}=N}}2^{k_1}\binom{N}{k_1,k_2,\cdots, k_{m-1}}.
\end{align}
\end{lemma}
The proof for Lemma \ref{multinomial identity II} relies on two different representation and expansion of $m^N$: one as $(\underbrace{1+\cdots+1}_\textrm{$N$})^N$ and the other as $(2+\underbrace{1+\cdots+1}_\textrm{$N-2$})^N$. See \cite{Stan} for more.

\subsection{Truncation}\label{Truncation subsection}
First, we would like to make some truncation to the sum in \eqref{formula for probability of log mod 1 of stick length between a and b}. We show that for any $\epsilon\in (0,1)$, the contribution to $F_N(a,b)$ of multinomial coefficients with $k_1+k_2< N^\epsilon$ is negligible. This allows us to only consider stick lengths $A_{k_1,k_2,\dots,k_m}$ where $k_1+k_2$ are sufficiently large.
\begin{prop}\label{multinomial sum tail bound}
For any $0<\epsilon<1$,
\begin{align}
\sum_{\substack{0\leq k_1,k_2\dots,k_m\leq N \\ k_1+k_2+\cdots+k_m=N \\ k_1+k_2< N^\epsilon}}\binom{N}{k_1,k_2,\dots,k_m} \ \leq \ (m-2)^NN^{N^\epsilon+2}.
\end{align}
\end{prop}
\begin{proof}
We know that if $k_1+k_2+\cdots+k_m=N$, then
\begin{align}
\binom{N}{k_1,k_2,\dots,k_m} 
&\ = \ \frac{N(N-1)\cdots (N-(k_1+k_2)+1)}{k_1!k_2!}\cdot \frac{(N-(k_1+k_2))!}{k_3!k_4!\cdots k_m!} \nonumber\\
&\ \leq \ N^{k_1+k_2}\cdot \binom{N-(k_1+k_2)}{k_3,k_4,\dots, k_m}.
\end{align}
Hence,
\begin{align}
&\sum_{\substack{0\leq k_1,k_2,\dots,k_m\leq N \\ k_1+k_2+\cdots+k_m=N \\ k_1+k_2< N^\epsilon}}\binom{N}{k_1,k_2,\dots,k_m} \nonumber\\
&\ = \ \sum_{\substack{0\leq k_1,k_2< N^\epsilon \\ k_1+k_2< N^\epsilon}}N^{k_1+k_2}\sum_{\substack{0\leq k_3,k_4,\dots,k_m\leq N-(k_1+k_2) \\ k_3+k_4+\cdots+k_m=N-(k_1+k_2)}}\binom{N-(k_1+k_2)}{k_3,k_4,\dots,k_m} \nonumber\\
&\ \leq \ (m-2)^N\sum_{\substack{0\leq k_1,k_2< N^\epsilon\\ k_1+k_2< N^\epsilon}}N^{k_1+k_2} \nonumber\\
&\ \leq \ (m-2)^N (N^\epsilon)^2 N^{N^\epsilon} \ \leq \ (m-2)^N N^{N^\epsilon+2}.
\end{align}
\end{proof}
We know the probability that a stick length $p_1^{k_1}p_2^{k_2}\cdots p_m^{k_m}$ satisfies $k_1+k_2< N^\epsilon$ is
\begin{align}\label{multinomial sum probability}
\frac{1}{m^N}\sum_{\substack{0\leq k_1,k_2\dots,k_m\leq N \\ k_1+k_2+\cdots+k_m=N \\ k_1+k_2< N^\epsilon}}\binom{N}{k_1,k_2,\dots,k_m}.
\end{align}
According to Proposition \ref{multinomial sum tail bound}, we have that \eqref{multinomial sum probability} is bounded above by $((m-2)/m)^NN^{N^\epsilon+2}$.
Hence, the logarithm of the probability is bounded above by
\begin{align}
\log\left(\left(\frac{m-2}{m}\right)^NN^{N^\epsilon+2}\right) \ = \ N\log\left(\frac{m-2}{m}\right)+(N^\epsilon+2)\log(N).
\end{align}
It is clear that the above goes to $-\infty$ as $N\rightarrow\infty$. Thus the probability goes to 0 as $N\rightarrow\infty$. So it suffices to consider the case where $k_1+k_2\geq N^\epsilon$. Now, fix $k(N)\geq N^\epsilon$ and let $0\leq k_1,k_2\leq k(N)$ with $k_1+k_2=k(N)$ and $0\leq k_3,k_4,\dots,k_m\leq N-k(N)$ with $k_3+k_4+\dots+k_m=N-k(N)$. By \cite{dependent}, we know that the frequency $k_1$ follows a binomial distribution with mean $k(N)/2$ and standard deviation $\sqrt{k(N)}/2$. Pick some $\delta\in (0,\epsilon/10)$. We see that it suffices to consider cases where $|k_1-k(N)/2|< (\lceil N^\delta\rceil \sqrt{k(N)})/2$, because the probability that $k_1$ is outside this range is asymptotically small by Chebyshev's inequality:
\begin{align}\label{stick model Chebyshev truncation}
\mathbb{P}\left(\left|k_1-\frac{k(N)}{2}\right| \ \geq \ \frac{\lceil N^\delta\rceil \sqrt{k(N)}}{2}\right) \ \leq \ \frac{1}{\lceil N^{\delta}\rceil^2}.
\end{align}

\subsection{Near uniform probability within small intervals} We keep the same notation and definition and fix $\epsilon$, $\delta$, $k(N):=k_1+k_2$, $k_3,k_4,\dots, k_m$ as before. Our goal is to prove that the logarithms of the stick length $A^{(N)}_{k_1,k_2,\dots,k_m}$ is roughly uniformly distributed over small intervals of $k_1$. Since $\lceil N^\delta\rceil=o(\lceil N^\delta\rceil\lceil\sqrt{k(N)}/2\rceil)$, we can evenly divide the range of $k_1$ between $k(N)/2\pm \lceil N^\delta\rceil\lceil\sqrt{k(N)}/2\rceil$ into intervals of $\lceil N^\delta \rceil$ values of $k_1$, where the $\ell$\textsuperscript{th} interval starts with $k_{1,\ell}=k_{1,\ell,0}$ and ranges over $k_{1,\ell,i}$ for $0\leq i\leq \lceil N^\delta\rceil-1$, as defined below: 
\begin{align}
&k_{1,\ell} \ := \ f_\ell\left(\frac{k(N)}{2}\right)+\ell \lceil N^\delta \rceil, \nonumber\\
&k_{1,\ell,i} \ := \ f_\ell\left(\frac{k(N)}{2}\right)+\ell \lceil N^\delta \rceil+i, \quad 0\leq i\leq N^\delta-1,
\end{align}
where $f_\ell(\cdot):=\lceil\cdot\rceil$ when $\ell\geq 0$ and $f_\ell(\cdot):=\lfloor\cdot\rfloor$ when $\ell<0$. We see that $\ell$ ranges from $-\lceil\sqrt{k(N)}/2\rceil$ to $\lceil\sqrt{k(N)}/2\rceil$. Note that we are using the floor and the ceiling functions to ensure that $f_{1,\ell,i}$'s have integer values. For convenience, we will make a slight abuse of notation to drop the floor and the ceiling signs from now on, as they have negligible effect on our calculation. We want to show that the difference $\left|\binom{k(N)}{k_{1,\ell,i}}-\binom{k(N)}{k_{1,\ell,j}}\right|$ is asymptotically smaller than $\binom{k(N)}{k_{1,\ell}}$ uniformly for all $0\leq i<j\leq N^\delta-1$, which would imply that the stick length is roughly uniformly distributed over small intervals of $k_1$. Since the binomial distribution is symmetric around its mean, it suffices to look at $\ell\geq 0$. Moreover, the probability density function is monotonically decreasing to the right of the mean, the difference is uniformly bounded within each interval and
\begin{align}
\left|\binom{k(N)}{k_{1,\ell,i}}-\binom{k(N)}{k_{1,\ell,j}}\right| &\ \leq \ \left|\binom{k(N)}{k_{1,\ell}}-\binom{k(N)}{k_{1,\ell+1}}\right|.
\end{align}
We follow Section (5.52) of \cite{dependent} to obtain a bound for the difference.
\begin{prop}\label{binomial difference bound}
For $\ell\leq \sqrt{k(N)}/2$, 
\begin{align}
\left|\binom{k(N)}{k_{1,\ell}}-\binom{k(N)}{k_{1,\ell+1}}\right| \ \leq \  O\left(\binom{k(N)}{k_{1,\ell}}\cdot N^{-\frac{3\epsilon}{10}}\right).
\end{align}
\end{prop}
Since the proof follows a similar argument to Section (5.52) of \cite{dependent}, we leave the proof details to \ref{appendix binomial difference bound}.

\subsection{Equidistribution within small intervals}\label{stick model equidistribution subsection} In this subsection, we want to show that for fixed $k(N), k_3, \dots, k_m$, the logarithm of the stick length $\log_B(A^{(N)}_{k_{1,\ell,i}, k_2, \dots, k_m})$ for $0\leq i\leq \lceil N^\delta\rceil-1$ converges to being equidistributed mod 1, in the sense that the logarithm of the stick length is thought of as a random variable that takes value in $\log_B(A^{(N)}_{k_{1,\ell,i},k_2,\dots, k_m})$ for $0\leq i\leq \lceil N^\delta \rceil-1$ with equal probabilities. We first state the following theorem which provides an easy criterion for checking equidistribution mod 1.

\begin{corollary}\cite{power saving}\label{equidistribution irrationality condition}
For any $a\in\R$, the arithmetic progression $\{a+n d\}_{n=1}^\infty$ is equidistributed mod 1 if and only if $d\not\in\Q$.
\end{corollary}

%\begin{lemma}[Weyl's Criterion]\cite{W}\label{Weyl}
%A sequence $\{a_n\}_{n=1}^\infty$ is equidistributed mod $1$ if and only if for all nonzero integers $\ell$,
%\begin{align}
%\lim_{n\rightarrow\infty}\frac{1}{n}\sum_{j=1}^n e^{2\pi i\ell a_j} \ = \ 0.
%\end{align}
%\end{lemma}
Now, note that for fixed $k(N)$ and $\ell$,
\begin{align}
&\log_B(A^{(N)}_{k_{1,\ell,i}, k_2, \dots, k_m}) \nonumber\\
&\ = \ k_{1,\ell,i}\log_B\left(\frac{p_1}{p_2}\right)+(k_1+k_2)\log_B\left(\frac{p_2}{p_3}\right)+\cdots+\left(\sum_{j=1}^{m-1}k_j\right)\log_B\left(\frac{p_{m-1}}{p_m}\right)+N\log(p_m) \nonumber\\
&\ = \ k_{1,\ell,i}\log_B\left(\frac{p_1}{p_2}\right)+k(N)\log_B\left(\frac{p_2}{p_3}\right)+\cdots+\left(\sum_{j=1}^{m-1}k_j\right)\log_B\left(\frac{p_{m-1}}{p_m}\right)+N\log(p_m)
\end{align}
forms an arithmetic progression over the $\ell$\textsuperscript{th} interval of $k_1$, i.e., $k_{1,\ell,i}$ for $0\leq i\leq N^{\delta}-1$, with common difference $\log_B(p_1/p_2)\notin\Q$. Then by Corollary \ref{equidistribution irrationality condition}, $\log_B(A^{(N)}_{k_{1,\ell,i},k_2, \dots, k_m})$ converges to equidistribution mod 1. However, Corollary \ref{equidistribution irrationality condition} does not provide any quantitative measure of the discrepancy from uniform distribution on $[0,1]$. It turns out that the key ingredient behind this quantification is irrationality exponent. We restate the definition here.

%Since the geometric progression consists of $N^\delta$ elements, which goes to $\infty$ as $N\rightarrow\infty$, and $\log_{B}(p_1/p_2)\not\in\Q$ by our assumption, then \cite{power saving} gives us the size of the set $J_{\ell}(a,b)\subset \{0,1,\dots, N^\delta-1\}$ of indices $i$ such that $\log_{B}(A^{(N)}_{k_{1,\ell,i},k_2,\dots,k_m})$ mod 1 is in $(a,b)$. We start with the definition of irrationality exponent, which we shall use to characterize the size of $J_{\ell}(a,b)$.

\begin{definition}\cite{power saving}
Suppose that $x$ is a real number. The \textbf{irrationality exponent} $\mu_x$ of $x$ is the supremum of the set of $\mu$ such that $0<|x-p/q|<1/q^\mu$ is satisfied by an infinite number of coprime integer pairs $(p,q)$ with $q>0$. If such a set does not exist, then we say $x$ has irrationality exponent $\infty$.
\end{definition}

Since $\Q$ is dense in $\R$, then every real number can be approximated by rational numbers. However, what differentiates them is how well they can be approximated. Irrationality exponent measures exactly how well a real number can be approximated by rational numbers. The bigger the irrationality exponent, the finer the rational approximation can be. The following are some well-known facts about irrationality exponent, which comes handy later in this section.

\begin{prop}\cite{power saving}
The irrationality exponent of any rational number is 1. As a consequence of Dirichlet's approximation theorem, which states that for any irrational number $x$,
\begin{align}
\left|x-\frac{p}{q}\right| \ < \ \frac{1}{q^2}
\end{align}
for infinitely $(p,q)$ where $p$ and $q$ are coprime, the irrationality exponent of any irrational number is at least 2. Further more, any algebraic irrational (irrational numbers that are zeros of polynomials over $\Q$) has irrationality exponent exactly 2.
\end{prop}

We now state a theorem that quantifies the discrepancy of the logarithm of the stick length in terms of irrationality exponent.

\begin{theorem}\cite[Theorem 3.2]{power saving}\label{power saving theorem}
Let $\kappa$ be the irrationality exponent of $\log_{B}(p_1/p_2)$ and $J_{\ell}(a,b)\subset \{0,1,\dots, N^\delta-1\}$ be the set of indices $i$ such that $\log_{B}(A^{(N)}_{k_{1,\ell,i},k_2,\dots,k_m})$ mod 1 is in $(a,b)$. If $\kappa<\infty$, then
\begin{align}
|J_\ell(a,b)| \ = \ (b-a)N^\delta+O\left(N^{\delta\left(1-\frac{1}{\kappa-1}+\epsilon'\right)}\right),
\end{align}
for every $\epsilon'>0$ and there is a power saving with the error term. The error term is optimal in the sense that
\begin{align}
\left|\frac{|J_\ell(a,b)|}{N^\delta}-(b-a)\right| \ = \ \Omega\left(N^{\delta(-\frac{1}{\kappa-1})-\epsilon'}\right)
\end{align}
for every $\epsilon'>0$. If $\kappa=\infty$, then
\begin{align}
|J_\ell(a,b)| \ = \ (b-a)N^\delta+o(N^\delta).
\end{align}
\end{theorem}
Note that \cite{power saving} has a different convention for irrationality exponent and if the irrationality exponent of $x$ is $\mu$ in our definition, then it is $\mu-1$ in their definition. This is the reason for the change made to the presentation of Theorem \cite{power saving}, where there is a $-1/(\kappa-1)$ instead of a $1/\kappa$ in the exponent.

\subsection{Evaluation of sums within and across intervals} We first proceed to the case where $\log_{B}(p_1/p_2)$ has a finite irrationality exponent $\kappa<\infty$. For fixed $k(N):=k_1+k_2,k_3,k_4,\dots k_m$, and $\ell\leq \sqrt{k(N)}/2$, we first count the number of $A^{(N)}_{k_{1,\ell,i},k_2,\dots,k_m}$  within the $\ell$\textsuperscript{th} interval such that $\log_{B}(A^{(N)}_{k_{1,\ell,i},k_2,\dots,k_m})$ mod 1 is in $(a,b)$. By definition, this is given by
\begin{align}
\sum_{i\in J_\ell(a,b)}\binom{k(N)}{k_{1,\ell,i}}.
\end{align}
By Proposition \ref{binomial difference bound} and Section \ref{stick model equidistribution subsection}, we have
\begin{align}\label{binomial sum over interval w power saving}
\sum_{i\in J_\ell(a,b)}\binom{k(N)}{k_{1,\ell,i}} &\ = \ \sum_{i\in J_\ell(a,b)}\Bigg\{\binom{k(N)}{k_{1,\ell}}+O\left(\binom{k(N)}{k_{1,\ell}}N^{-\frac{3\epsilon}{10}}\right)\Bigg\} \nonumber\\
&\ = \ \Bigg\{\binom{k(N)}{k_{1,\ell}}\sum_{i\in J_\ell(a,b)}1\Bigg\}+O\left(\binom{k(N)}{k_{1,\ell}}N^{-\frac{3\epsilon}{10}}\sum_{i\in J_\ell(a,b)} 1\right) \nonumber \\
&\ = \ (b-a)N^\delta\binom{k(N)}{k_{1,\ell}}+O\left(\binom{k(N)}{k_{1,\ell}}N^{\delta\left(1-\frac{1}{\kappa-1}+\epsilon'\right)}\right).
\end{align}
We justify the error term in the last line of \eqref{binomial sum over interval w power saving}. Since $\kappa$ is the irrationality exponent of an irrational number, then $\kappa\geq 2$ (see \cite{BBS}). Also, $\epsilon'>0$, so $-1<-1/(\kappa-1)+\epsilon'$. Moreover, for some proper choice of $\epsilon'$, we have $-1/(\kappa-1)+\epsilon'<0$. Combining these with the fact that $\delta\in (0,\epsilon/10)$ gives the dominant error term in the last line of \eqref{binomial sum over interval w power saving}.

Now, we count the number of stick lengths $A^{(N)}_{k_{1,\ell,i},k_2,\dots,k_m}$ such that $\log_{B}(A^{(N)}_{k_{1,\ell,i},k_2,\dots, k_m})$ mod 1 is in $(a,b)$ over all the intervals. By the truncation in Section \ref{Truncation subsection} the main term comes from the sum over $\ell$ from $-\sqrt{k(N)}/2$ to $\sqrt{k(N)}/2$. The number is given by
\begin{align}\label{binomial sum over all intervals}
\sum_{\ell=-k(N)/(2N^\delta)}^{k(N)/(2N^\delta)}\sum_{i\in J_\ell(a,b)}\binom{k(N)}{k_{1,\ell,i}} \ = \ \sum_{\ell=-\sqrt{k(N)}/2}^{\sqrt{k(N)}/2}\sum_{i\in J_\ell(a,b)}\binom{k(N)}{k_{1,\ell,i}}+\sum_{|\ell|>\sqrt{k(N)}/2}\sum_{i\in J_\ell(a,b)}\binom{k(N)}{k_{1,\ell,i}}.
\end{align}
Using \eqref{stick model Chebyshev truncation}, we can provide an upper bound on the second sum on the RHS of \eqref{binomial sum over all intervals}:
\begin{align}\label{truncated error term for binomial sum over all intervals}
\sum_{|\ell|>\sqrt{k(N)}/2}\sum_{i\in J_\ell(a,b)}\binom{k(N)}{k_{1,\ell,i}}\nonumber &\ \leq \
\mathbb{P}\left(\left|k_1-\frac{k(N)}{2}\right| \ \geq \ \frac{N^\delta\sqrt{k(N)}}{2}\right)\cdot \sum_{0\leq k_1\leq k(N)}\binom{k(N)}{k_1} \nonumber\\
&\ = \ \mathbb{P}\left(\left|k_1-\frac{k(N)}{2}\right| \ \geq \ \frac{N^\delta\sqrt{k(N)}}{2}\right)\cdot 2^{k(N)} \ = \ O\left(\frac{2^{k(N)}}{N^{2\delta}}\right).
\end{align}
On the other hand, we apply \eqref{binomial sum over interval w power saving} to the first sum on the RHS of \eqref{binomial sum over all intervals}:
\begin{align}\label{truncated main term for binomial sum over all intervals with error}
\sum_{\ell=-\sqrt{k(N)}/2}^{\sqrt{k(N)}/2}\sum_{i\in J_\ell(a,b)}\binom{k(N)}{k_{1,\ell,i}} &\ = \ \sum_{\ell=-\sqrt{k(N)}/2}^{\sqrt{k(N)}/2}(b-a)N^\delta\binom{k(N)}{k_{1,\ell}} \nonumber\\
&\quad\quad+O\left(\sum_{\ell=-\sqrt{k(N)}/2}^{\sqrt{k(N)}/2}\binom{k(N)}{k_{1,\ell}}N^{\delta\left(1-\frac{1}{\kappa-1}+\epsilon'\right)}\right).
\end{align}
We know that
\begin{align}\label{remove truncation from truncated main term for binomial sum over all intervals}
\sum_{\ell=-\sqrt{k(N)}/2}^{\sqrt{k(N)}/2}\binom{k(N)}{k_{1,\ell}} &\ = \ \sum_{\ell=-k(N)/(2N^{\delta})}^{k(N)/(2N^{\delta})}\binom{k(N)}{k_{1,\ell}}+O\left(\sum_{|\ell|>\sqrt{k(N)}/2}\binom{k(N)}{k_{1,\ell}}\right) \nonumber\\
&\ = \ \sum_{\ell=-k(N)/(2N^{\delta})}^{k(N)/(2N^{\delta})}\binom{k(N)}{k_{1,\ell}}+O\left(\frac{2^{k(N)}}{N^{2\delta}}\right).
\end{align}
By (5.17) of \cite{dependent}, by choosing $q=N^\delta$, we have 
\begin{align}\label{application of multisection formula to binomial sum}
\sum_{\ell=-k(N)/(2N^{\delta})}^{k(N)/(2N^{\delta})}\binom{k(N)}{k_{1,\ell}} \ = \ \frac{2^{k(N)}}{N^{\delta}}+O\left(2^{k(N)}e^{-3k(N)^{1-2\delta}}\right)
\end{align}
Substitute \eqref{application of multisection formula to binomial sum} into \eqref{remove truncation from truncated main term for binomial sum over all intervals} and then substitute the resulting formula into \eqref{truncated main term for binomial sum over all intervals with error}, we have
\begin{align}\label{main term for binomial sum over all intervals}
\sum_{\ell=-\sqrt{k(N)}/2}^{\sqrt{k(N)}/2}\sum_{i\in J_\ell(a,b)}\binom{k(N)}{k_{1,\ell,i}} &\ = \ (b-a)2^{k(N)}+O\left(2^{k(N)}N^{-\delta}+2^{k(N)}N^{\delta\left(-\frac{1}{\kappa-1}+\epsilon'\right)}\right) \nonumber\\
 &\ = \ (b-a)2^{k(N)}+O\left(2^{k(N)}N^{\delta\left(-\frac{1}{\kappa-1}+\epsilon'\right)}\right),
\end{align}
where the last line follows from $-1<-1/(\kappa-1)+\epsilon'<0$. Substituting \eqref{truncated error term for binomial sum over all intervals} and \eqref{main term for binomial sum over all intervals} into \eqref{binomial sum over all intervals} yields
\begin{align}\label{final estimate for binomial sum over all intervals}
\sum_{\ell=-k(N)/(2N^\delta)}^{k(N)/(2N^\delta)}\sum_{i\in J_\ell(a,b)}\binom{k(N)}{k_{1,\ell,i}} \ = \ (b-a)2^{k(N)}+O\left(2^{k(N)}N^{\delta\left(-\frac{1}{\kappa-1}+\epsilon'\right)}\right).
\end{align}
Finally, we sum over all $k(N),k_3,k_4,\dots,k_m$, which gives us the total number of lengths $A^{(N)}_{k_1,k_2,\dots, k_m}$ such that $\log_{B}(A^{(N)}_{k_1,k_2,\dots, k_m})$ mod 1 is in $(a,b)$. We break the sum based on whether $k(N)\geq N^\epsilon$ or not. For $k(N)\geq N^\epsilon$, we apply \eqref{final estimate for binomial sum over all intervals}:

\begin{align}\label{main term for multinomial sum}
&\sum_{\substack{0\leq k(N),k_3,k_4,\dots,k_m\leq N \\ k(N)+k_3+k_4+\cdots+k_m=N \\ k(N)\geq N^\epsilon}}\binom{N}{k(N),k_3,k_4,\dots,k_m}\sum_{\ell=-k(N)/(2N^\delta)}^{k(N)/(2N^\delta)}\sum_{i\in J_\ell(a,b)}\binom{k(N)}{k_{1,\ell,i}} \nonumber\\
&\ = \ (b-a)\sum_{\substack{0\leq k(N),k_3,k_4,\dots,k_m\leq N \\ k(N)+k_3+k_4+\cdots+k_m=N \\ k(N)\geq N^\epsilon}}2^{k(N)}\binom{N}{k(N),k_3,k_4,\dots,k_m} \nonumber\\
&\quad\quad+O\left(N^{\delta\left(-\frac{1}{\kappa-1}+\epsilon'\right)}\sum_{\substack{0\leq k(N),k_3,k_4,\dots,k_m\leq N \\ k(N)+k_3+k_4+\cdots+k_m=N \\ k(N)\geq N^\epsilon}}2^{k(N)}\binom{N}{k(N),k_3,k_4,\dots,k_m}\right).
\end{align}
Note that by Lemma \ref{multinomial identity II}
\begin{align}\label{difference between two estimate on multinomial sum}
&\sum_{\substack{0\leq k(N),k_3,k_4,\dots,k_m\leq N \\ k(N)+k_3+k_4+\cdots+k_m=N \\ k(N)\geq N^\epsilon}}2^{k(N)}\binom{N}{k(N),k_3,k_4,\dots,k_m} \nonumber\\
&\ = \ m^N-\sum_{\substack{0\leq k_1,k_2,\dots, k_m\leq N \\ k_1+k_2+\cdots+k_m=N \\ k_1+k_2<N^\epsilon}}\sum_{k_1=0}^{k_1+k_2}\binom{N}{k_1,k_2,\dots, k_m} \nonumber\\
&\ = \ m^N+O((m-2)^N N^{N^\epsilon+2}),
\end{align}
where in the last line, we apply Proposition \ref{multinomial sum tail bound}. Substituting \eqref{difference between two estimate on multinomial sum} into \eqref{main term for multinomial sum} gives
\begin{align}\label{main term of multinomial sum final estimate}
&\sum_{\substack{0\leq k(N),k_3,k_4,\dots,k_m\leq N \\ k(N)+k_3+k_4+\cdots+k_m=N \\ k(N)\geq N^\epsilon}}\binom{N}{k(N),k_3,k_4,\dots,k_m}\sum_{\ell=-k(N)/(2N^\delta)}^{k(N)/(2N^\delta)}\sum_{i\in J_\ell(a,b)}\binom{k(N)}{k_{1,\ell,i}} \nonumber\\
&\ = \ (b-a)m^N+O\left((m-2)^NN^{N^{\epsilon}+2}+m^NN^{\delta\left(-\frac{1}{\kappa-1}+\epsilon'\right)}+(m-2)^NN^{N^{\epsilon}+2}N^{\delta\left(-\frac{1}{\kappa-1}+\epsilon'\right)}\right) \nonumber\\
&\ = \ (b-a)m^N+O\left(m^NN^{\delta\left(-\frac{1}{\kappa-1}+\epsilon'\right)}\right).
\end{align}
For $k(N)<N^\epsilon$, by Lemma \ref{multinomial identity I} and Proposition \ref{multinomial sum tail bound} we have
\begin{align}\label{error term of multinomial sum} 
&\sum_{\substack{0\leq k(N),k_3,k_4,\dots,k_m\leq N \\ k(N)+k_3+k_4+\cdots+k_m=N \\ k(N)< N^\epsilon}}\binom{N}{k(N),k_3,k_4,\dots,k_m}\sum_{\ell=-k(N)/(2N^\delta)}^{k(N)/(2N^\delta)}\sum_{i\in J_\ell(a,b)}\binom{k(N)}{k_{1,\ell,i}} \nonumber\\
&\ \leq \ \sum_{\substack{0\leq k(N),k_3,k_4,\dots,k_m\leq N \\ k(N)+k_3+k_4+\cdots+k_m=N \\ k(N)< N^\epsilon}}\binom{N}{k(N),k_3,k_4,\dots,k_m}\sum_{k_1=0}^{k(N)}\binom{k(N)}{k_1} \nonumber\\
&\ = \ \sum_{\substack{0\leq k_1,k_2,\dots,k_m\leq N \\ k_1+k_2+\cdots+k_m=N \\ k_1+k_2< N^\epsilon}}\binom{N}{k_1,k_2,\dots,k_m} \ = \ O\left((m-2)^NN^{N^{\epsilon}+2}\right).
\end{align}
Combining \eqref{main term of multinomial sum final estimate} and \eqref{error term of multinomial sum}, we have
\begin{align}
&\sum_{\substack{0\leq k(N),k_3,k_4,\dots,k_m\leq N \\ k(N)+k_3+k_4+\cdots+k_m=N}}\binom{N}{k(N),k_3,k_4,\dots,k_m}\sum_{\ell=-k(N)/(2N^\delta)}^{k(N)/(2N^\delta)}\sum_{i\in J_\ell(a,b)}\binom{k(N)}{k_{1,\ell,i}} \nonumber\\
&\ = \ (b-a)m^N+O\left(N^{\delta\left(-\frac{1}{\kappa-1}+\epsilon'\right)}m^N+N^{N^\epsilon+2}(m-2)^N\right) \nonumber\\
&\ = \ (b-a)m^N+O\left(N^{\delta\left(-\frac{1}{\kappa-1}+\epsilon'\right)}m^N\right),
\end{align}
After dividing by $m^N$, the total number of sticks after stage $N$, we arrive at
\begin{align}\label{stick model final estimate w power saving}
F_N(a,b) \ = \ b-a+O\left(N^{\delta\left(-\frac{1}{\kappa-1}+\epsilon'\right)}\right).
\end{align}

Thus, we have shown that if $\log_B(p_1/p_2)$ has a finite irrationality exponent $\kappa$, then the logarithm of the stick lengths is equidistributed mod 1, and thus the stick lengths converge to the Benford distribution. Moreover, the discrepancy between the distribution of the logarithm of the stick length and uniform distribution mod 1 is quantified in terms of the irrationality exponent of $\log_B(p_1/p_2)$. Equation \eqref{stick model final estimate w power saving} tells us that the smaller the irrationality exponent, the smaller the discrepancy is (thanks to the guarantee that the error term is optimal), and the closer the distribution of the stick length is to the Benford distribution. If $\log(p_1/p_2)$ has infinite irrationality exponent, the calculation is exactly the same, except that we start with the error term $o(N^\delta)$ instead of $O(N^{\delta\left(1-\frac{1}{\kappa-1}+\epsilon'\right)})$ in Section \ref{stick model equidistribution subsection}, and so we end up with the $F_N(a,b) \ = \ b-a+o(1)$.

Thus, this establishes the strong Benford behavior of fixed multi-proportion 1-dimensional stick fragmentation process when $y_i\not\in\Q$ for some $1\leq i\leq m-1$. On a final note, recall from Section \ref{prelim subsection} our choices of the order of $p_1, p_2, \dots, p_m$ and thus factorization are arbitrary. Hence, to obtain an optimal error term, we simply need to choose an order of $p_1, p_2, \dots, p_m$ such that $\log_B(p_1/p_2)$ is irrational with the least irrationality exponent. Let $\kappa_0$ be such irrationality exponent. Then we have
\begin{align}
F_N(a,b) \ = \ b-a+O\left(N^{\delta\left(-\frac{1}{\kappa_0-1}+\epsilon'\right)}\right).
\end{align}
\qed

\section{Proof of Theorem \ref{box fragmentation conj answer arbitrary dim}}\label{higher dimension}

\subsection{Preliminary}\label{prelim of proof of box fragmentation theorem}
In this section, we prove Theorem \ref{box fragmentation conj answer arbitrary dim}, i.e., Conjecture 4.1 of \cite{box fragmentation}, under some conditions. Let us first recall the conjecture.
\begin{conj: box fragmentation conjecture}
Every linear fragmentation process satisfies the Maximum Criterion in all dimensions $1\leq d\leq m$.
\end{conj: box fragmentation conjecture}

Betti et al. \cite{box fragmentation} were able to prove the Conjecture for any arbitrary $m\geq 1$ and $d=1$. That is, they showed that under any linear fragmentation process, the length of the longest side of a fragmented box of any arbitrary dimension converges to strong Benford behavior. For the rest of this paper, we prove Theorem \ref{box fragmentation conj answer arbitrary dim} for any $m\geq 1$ and $d\leq m$, i.e., we consider the volume of the largest $d$-dimensional face of a fragmented box of any arbitrary dimension $m\geq d$ converges to strong Benford behavior. Our proof relies on the following conditions on the distribution of $P_i^{(j)}$:
\begin{enumerate}
\item Condition 1: $P_i^{(j)}$'s are i.i.d. for all $1\leq i\leq m$ and $1\leq j\leq N$. That is, $P_i^{(j)}$'s are all independent and there exists a random variable $P$ such that $P_i^{(j)}\sim P$ for all $i,j$;
\item Condition 2: Let the mean and the variance of $\log_B(P)$ be $\mu_P$ and $\sigma_P$. Then $\log_B(P)$ has finite centered moments, i.e., $\mathbb{E}[(\log_B(P)-\mu_P)^k]<\infty$ for all $k\geq 1$ and the characteristic function of $(\log_B(P)-\mu_P)/\sigma_P$ admits Taylor expansion. Moreover, the variance is nonzero;
\item Condition 3: Let $f(x)$ be the probability density function of $\log_B(P)$, then we have $\textup{ess sup}_x f(x)<\infty$;
\item Condition 4: $f(x)$ is differentiable with $\textup{sup}_x |f'(x)|<\infty$.
\end{enumerate}
%\item Condition 2: $\log_B(P_i)$ is finitely supported, i.e. there exists a constant $C>0$ such that $\log_B(P_i)$ takes value in $[-C,C]$;
%\item Condition 3: Suppose that $P^{(k)}_i$ is a sequence of i.i.d. random variables $\sim P_i$. Define the random variable
%\begin{align}\label{CLT equation}
%Z_i^{(N)} \ := \ \frac{\log_B(P_i^{(1)}\cdots P_i^{(N)})-N\mu_{P}}{\sqrt{N}\sigma_{P}}.
%\end{align}
%By CLT, $Z_i^{(N)}$ converges to the standard normal $N(0,1)$. Additionally, assume that $Z_i^{(N)}$ has probability density function $f_{Z_i^{(N)}}(z_i)=\varphi(z_i)+A(z_i)$ and cumulative density function $F_{Z_i^{(N)}}(z_i)=\Phi(z_i)+B(z_i)$, where $\varphi(x)$ and $\Phi(x)$ are the probability density function and cumulative density function of the standard normal $N(0,1)$, $A(x)=O(N^{-1/2-\delta})$ and $B(x)=O(N^{-\delta})$ for some $\delta>0$.

These conditions ensure that we are working with ``nice'' proportion cuts and are satisfied by a wide range of random variables. Moreover, we believe that Theorem \ref{box fragmentation conj answer arbitrary dim} should hold in a far more general setting, particular when the $P_i^{(j)}$'s have independent but non-identical distributions. For example, as long as $\{P_i^{(j)}\}_{j=1}^N$ satisfies the conditions for CLT for every $i$, then the distribution of the normalized side lengths of the box are still approximately Gaussian, and the proof techniques developed in this paper should extend without major difficulty. The only technical caveat is perhaps that we would need to use \cite{Vaughan} instead of \eqref{order statistics joint PDF formula} to obtain a formula for the volume of the largest lower dimensional face of the box, since the side lengths are no longer i.i.d.. The formula in \cite{Vaughan} is rather complicated--it involves the permanent of a matrix whose entries are simple expressions in terms of the PDF and the CDF of random variables. However, after permanent expansion, it is essentially a linear sum with terms resembling \eqref{order statistics joint PDF formula}. Hence, we believe that our analysis should extend to the general setting where $P_i^{(j)}$ are independent but non-identical without too much trouble, but for the clarity of exposition, we decide to restrict ourselves to the i.i.d. setting.

Under these conditions, our goal is then to prove that the volume $\mathfrak{m}_d^{(N)}$ of the largest $d$-dimensional face of a $m$-dimensional fragmented box converges to strong Benford behavior. Once we prove this, then by the Maximum Criterion, the volume $V_d^{(N)}$ of the $d$-dimensional faces of a $m$-dimensional fragmented box also converges to strong Benford behavior. To begin, let us first introduce the notion of order statistics.

\begin{definition}
Suppose that $X_1,X_2,\dots, X_k$ are random variables, and $[]_i$ returns the $i$\textsuperscript{th} largest number among a list $(a_1, a_2, \dots, a_k)$ of real numbers. For each outcome $\omega$ in the sample space $ \Omega$, we define
\begin{align}
X_{(k-i+1)}(w) \ := \ [(X_1(w), X_2(w),\dots, X_k(w))]_{k-i+1}.
\end{align}
We say that the random variable $X_{(k-i+1)}$ is \textbf{the $i$\textsuperscript{th} order statistics}, or the $i$\textsuperscript{th} largest random variable among $X_1, X_2, \dots, X_k$.
Hence, as random variables, $X_{(1)}\leq X_{(2)}\leq\dots\leq X_{(k)}$.
\end{definition}
Let $S_1^{(N)}, \cdots, S_m^{(N)}$ be the side lengths of the $m$-dimensional box after stage $N$. Then the volume of the largest $d$-dimensional face is the product of the longest $d$ side lengths of the box, i.e., $\mathfrak{m}_d^{(N)}=\prod_{i=1}^{d}S_{(m+1-i)}^{(N)}$. Since $S_i^{(N)}=P_i^{(1)}\cdots P_i^{(N)}$ for all $i$ and the $P_i^{(j)}$'s are between $0$ and $1$, then $\mathfrak{m}_d^{(N)}$ is asymptotically very small. To make this quantity more convenient to work with, we normalize it as follows. By Condition 1, $P_i^{(j)}$'s are i.i.d., which means that $S_i^{(N)}$'s are i.i.d. as well. Since $\mathbb{E}[\log_B(P)]=\mu_P$ and $\textup{Var}[\log_B(P)]=\sigma_P^2$, then for all $i$,
\begin{align}\label{log side length expected value and variance}
\mathbb{E}[\log_B(S_i^{(N)})] \ = \ N\mu_P\quad \textup{and} \quad \textup{Var}[\log_B(S_i^{(N)})] \ = \ N\sigma_P^2.
\end{align}
Now consider
\begin{align}\label{normalized log side length}
Z_i^{(N)} \ := \ \frac{\log_B(S_i^{(N)})-N\mu_P}{\sqrt{N}\sigma_P}.
\end{align}
It is a well-known fact from order statistics that if the order statistics of $(X_1, \dots, X_k)$ are $(X_{(1)}, \dots, X_{(k)})$, then the order statistics of $(g(X_1), \dots, g(X_k))$ are $(g(X_{(1)}), \dots, g(X_{(k)}))$ for any monotonically increasing function $g$. Since the order statistics of $(S_1^{(N)}, \dots, S_m^{(N)})$ are $(S_{(1)}^{(N)}, \dots, S_{(m)}^{(N)})$ and $g(x):=(\log_B(x)-N\mu_P)/\sqrt{N}\sigma_P$ is monotonically increasing, then it follows that the order statistics of $(g(S_1^{(N)}), \dots, g(S_m^{(N)}))$ are $(g(S_{(1)}^{(N)}), \dots, g(S_{(m)}^{(N)}))$, which are just $(Z^{(N)}_{(1)}, \dots, Z^{(N)}_{(m)})$. We are now ready to define the normalized form of $\mathfrak{m}_d^{(N)}$ with which we shall use to establish strong Benford behavior:
\begin{align}\label{normalized log max volume}
Y^{(N)} \ := \ \frac{\log_B(\mathfrak{m}_d^{(N)})-d\cdot N\mu_P}{\sqrt{N}\sigma_P} \ = \ \sum_{i=1}^{d}\frac{\log_B(S^{(N)}_{(m+1-i)})-N\mu_P}{\sqrt{N}\sigma_P} \ = \ \sum_{i=1}^{d}Z_{(m+1-i)}^{(N)}.
\end{align}
To find the PDF of $Y^{(N)}$, we need the joint PDF of $Z_{(m+1-i)}^{(N)}$ for all $1\leq i\leq d$, which is a standard question in order statistics.
\begin{prop}\cite{orderstats}\label{order statistics joint PDF proposition}
For i.i.d. random variables $X_1, \dots, X_n$ with PDF $f(x)$ and CDF $F(x)$, the joint PDF of $X_{(i)},\dots, X_{(n)}$ is given by
\begin{align}\label{order statistics joint PDF formula}
f_{X_{(i)},\dots, X_{(n)}}(x_i,\dots, x_n) \ = \ C_n^i f(x_i)\cdots f(x_n)(F(x_i))^{i-1},
\end{align}
where $1\leq i\leq n$ and $C_n^i:=n!/(i-1)!$.
\end{prop}

Since $Z_{i}^{(N)}$'s are i.i.d., then to find the joint PDF of $Z_{(m+-i)}^{(N)}$, it suffices to find the PDF and CDF of $Z_{i}^{(N)}$'s, which converge to that of the standard Gaussian distribution by CLT. However, to gain a sharp estimate on the distribution of $Y^{(N)}$, we would need to quantify the error terms of the PDF and the CDF, as follows.

To begin, we prove the following Riemann-Lebesgue lemma type bound on the characteristic function of a random variable.
\begin{lemma}\label{estimate for characteristic function}
Suppose that $f(x)$ is the PDF of a random variable $X$, where $f(x)$ is differentiable with $\sup_{x\in \R}|f'(x)|<\infty$ and $X$ has finite mean and nonzero, finite variance. Then the characteristic function
\begin{align}
\psi(t) \ := \ \int_{-\infty}^\infty f(x)e^{itx}dx
\end{align}
satisfies $\psi(t) = O(t^{-2/5})$.
\end{lemma}
\begin{proof}
We know that $e^{itx}=\cos(tx)+i\sin(tx)$. Hence, it suffices to show that
\begin{align}
\int_{-\infty}^\infty f(x)\cos(tx)dx \ = \ O(t^{-2/5}) \quad \textup{and} \quad \int_{-\infty}^\infty f(x)\sin(tx)dx \ = \ O(t^{-2/5}).
\end{align}
The proof for the two bounds are similar, and we only show it for $\int_{-\infty}^\infty f(x)\cos(tx)dx$. We first estimate the tail of the integral.
\begin{align}\label{estimate of tail of characteristic function}
\left|\int_{|x|\geq t^{1/5}}f(x)\cos(tx)dx\right| &\ \leq \ \int_{|x|\geq t^{1/5}}f(x)dx \nonumber\\
&\ = \ \mathbb{P}\left(|X|\geq t^{1/5}\right) \nonumber\\
&\ \leq \ \mathbb{P}\left(|X-\mathbb{E}[X]|\geq t^{1/5}-\mathbb{E}[X]\right)+\mathbb{P}\left(|X-\mathbb{E}[X]|\geq t^{1/5}+\mathbb{E}[X]\right) \nonumber\\
&\ \leq \ \frac{\textup{Var}[X]}{(t^{1/5}-\mathbb{E}[X])^2}+\frac{\textup{Var}[X]}{(t^{1/5}+\mathbb{E}[X])^2} \ = \ O(t^{-2/5}),
\end{align}
where the last line follows from Chebyshev's inequality, which requires $X$ to have finite mean and nonzero, finite variance. Hence, we can restrict the domain of integration to $[-t^{1/5},t^{1/5}]$. We know that if $g(x)$ is differentiable on $[a,b]$ with $\sup_{x\in [a,b]}|g'(x)|=M<\infty$ \cite{G}, then for any Riemann sum $S=\sum_{k=1}^n g(x_k^{*})(x_k-x_{k-1})$,
\begin{align}
\left|\int_{-\infty}^\infty g(x)dx-S\right| \ \leq \ \frac{M(b-a)\delta_{\textup{max}}}{2},
\end{align}
where $\delta_{max}=\max_{k}(x_k-x_{k-1})$. Now consider a Riemann sum $S_t=\sum_{k=1}^nf(x_k^{*})(x_k-x_{k-1})$ with the partition $x_{k}-x_{k-1}=t^{-3/5}$ for all $k$. Define the function $f_t$ on $[-t^{1/5},t^{1/5}]$ such that $f_t(x):=f(x_k^{*})$ for all $x\in [x_{k-1},x_k)$ and for every $k$. Since $f$ is nonnegative, then $f_t$ is also nonnegative. Hence, 
\begin{align}\label{estimate of difference between characteristic function and Riemann sum}
\left|\int_{-t^{1/5}}^{t^{1/5}}f(x)\cos(tx)dx-\int_{-t^{1/5}}^{t^{1/5}}f_t(x)\cos(tx)dx\right| &\ \leq \ \left|\int_{-t^{1/5}}^{t^{1/5}}f(x)dx-S_t\right| \nonumber\\
&\ \leq \ \frac{M_t\cdot (2t^{1/5})\cdot t^{-3/5}}{2} \ = \ O(t^{-2/5}),
\end{align}
where $M_t:=\sup_{x\in [-t^{1/5},t^{1/5}]}|f'(x)|<\infty$. Moreover, we have
\begin{align}
\int_{-t^{1/5}}^{t^{1/5}}f_t(x)\cos(tx)dx &\ = \ \int_{-t^{1/5}}^{t^{1/5}}\sum_{k=1}^n f(x_k^{*})\chi_{[x_{k-1},x_k)}(x)\cos(tx)dx \nonumber\\
&\ = \ \sum_{k=1}^n f(x_k^{*})\int_{x_{k-1}}^{x_k}\cos(tx)dx \nonumber\\
&\ = \ \frac{1}{t}\sum_{k=1}^n f(x_k^{*})(\sin(tx_{k})-\sin(tx_{k-1})) \nonumber\\
&\ = \ O\left(\frac{1}{t}\sum_{k=1}^n f(x_k^{*})\right).
\end{align}
Since $S_t=t^{-3/5}\sum_{k=1}^n f(x_k^{*})$ and that the Riemann sum is finite for large $t$, then
\begin{align}\label{estimate of Riemann sum}
 \int_{-t^{1/5}}^{t^{1/5}}f_t(x)\cos(tx)dx &\ = \  O(t^{-2/5}).  
\end{align}
Thus, combining \eqref{estimate of tail of characteristic function}, \eqref{estimate of difference between characteristic function and Riemann sum}, and \eqref{estimate of Riemann sum}, we have
\begin{align}
&\left|\int_{-\infty}^\infty f(x)\cos(tx)dx\right| \nonumber\\
&\ \leq \ \left|\int_{|x|\geq t^{1/5}} f(x)\cos(tx)dx\right|+\left|\int_{-t^{1/5}}^{t^1/5}f(x)\cos(tx)dx-\int_{-t^{1/5}}^{t^{1/5}}f_t(x)\cos(tx)dx\right| \nonumber\\
&\quad\quad +\left|\int_{-t^{1/5}}^{t^{1/5}}f_t(x)\cos(tx)dx\right| \ = \ O(t^{-2/5}).
\end{align}
\end{proof}

\begin{lemma}\label{PDF and CDF estimates for normalized log side length}
Let $Z_i^{(N)}\sim Z^{(N)}$. For any $0<\epsilon<1/2$, the PDF $f_{Z^{(N)}}$ and the CDF $F_{Z^{(N)}}$ of $Z^{(N)}$ satisfy
\begin{align}\label{PDF estimate for normalized log side length}
f_{Z^{(N)}}(z) \ = \ \varphi(z)+O(N^{-3/2+3\epsilon})
\end{align}
\begin{align}
F_{Z^{(N)}}(z) \ = \ \Phi(z)+O(N^{-1/2}),
\end{align}
where $\varphi$ and $\Phi$ are the PDF and the CDF of the standard Gaussian distribution.
\end{lemma}
\begin{proof}
We first prove \eqref{PDF estimate for normalized log side length}. Let the characteristic function of $W:=(\log_B(P)-\mu_P)/\sigma_P$ be $\psi_W(t)$, which has mean $0$, variance $1$, and finite moments by Condition 2. Hence,
\begin{align}
\psi_W(t) &\ = \ \sum_{k=0}^\infty \frac{\mathbb{E}[W^k](it)^k}{k!} \nonumber\\
&\ = \ 1-\frac{t^2}{2}+O(t^3)
\end{align}
for sufficiently small $t$. Let the characteristic function of $Z^{(N)}$ be $\psi_{Z^{(N)}}(t)$. By properties of characteristic functions,
\begin{align}\label{characteristic function of normalized log side length}
\psi_{Z^{(N)}}(t) &\ = \ (\psi_{W}(t/\sqrt{N}))^N \nonumber\\
&\ = \ \left(1-\frac{t^2}{2N}+O\left(\frac{t^3}{N^{3/2}}\right)\right)^N \nonumber\\
&\ = \ \textup{exp}\left[N\log\left(1-\frac{t^2}{2N}+O\left(\frac{t^3}{N^{3/2}}\right)\right)\right] \nonumber\\
&\ = \ \textup{exp}\left[N\left(\log\left(1-\frac{t^2}{2N}\right)+O\left(\frac{1}{1-t^3/N^{3/2}}\cdot\frac{t^3}{N^{3/2}}\right)\right)\right] \nonumber\\
&\ = \ \textup{exp}\left[N\left(-\frac{t^2}{2N}+O\left(\frac{t^4}{N^2}\right)+O\left(\frac{t^3}{N^{3/2}}\right)\right)\right] \nonumber\\
&\ = \ e^{-t^2/2}\left(1+O\left(\frac{t^3}{N^{3/2}}\right)\right),
\end{align}
where in the last two lines, we make use of the Taylor series expansions of $\log(1+x)$ and $e^x$. Note that the above estimate holds for $|t|\leq N^\epsilon$ for any $0<\epsilon<1/2$. By Fourier inversion, the PDF of $Z^{(N)}$ is
\begin{align}
\frac{1}{2\pi}\int_{-\infty}^\infty\psi_{Z^{(N)}}(t) e^{-itx}dt.
\end{align}
We shall split the integral above into several parts and provide an estimate for each. We first estimate the integral for $|t|\leq N^\epsilon$. By \eqref{characteristic function of normalized log side length},
\begin{align}\label{Fourier inversion in small regime}
\frac{1}{2\pi}\int_{|t|\leq N^\epsilon}\psi_{Z^{(N)}}(t)e^{-itx}dt \ = \ \varphi(x)+O(N^{-3/2+3\epsilon}).
\end{align}

%Let $\mathcal{F}(\cdot)$ be the Fourier transform. For any differentiable PDF $f$ satisfying $\int_{-\infty}^\infty |f'(x)|dx<\infty$, we have
%\begin{align}
%|\mathcal{F}(f)(k)|+|\mathcal{F}(f')(k)| &\ = \ \left|\int_{-\infty}^\infty f(x)e^{ikx}dx\right|+\left|\int_{-\infty}^\infty f'(x)e^{ikx}dx\right| \nonumber\\
%&\ \leq \ \int_{-\infty}^\infty f(x)dx+\int_{-\infty}^\infty |f'(x)|dx \nonumber\\
%&\ = \ 1+\int_{-\infty}^\infty |f'(x)|dx \ < \ \infty.
%\end{align}
%By integration by parts, we have
%\begin{align}
%\mathcal{F}(f')(k) &\ = \ \int_{-\infty}^\infty f'(x)e^{ikx}dx \nonumber\\
%&\ = \ f(x)e^{ikx}\Big|_{-\infty}^\infty -ik\int_{-\infty}^\infty f(x)e^{ikx}dx \nonumber\\
%&\ = \ -ik\mathcal{F}(f)(k),
%\end{align}
%where in the last line, we use the fact that for a PDF $f$, $f(-\infty)=f(\infty)=0$. Hence,
%\begin{align}
%|\mathcal{F}(f)(k)| \ \leq \ \frac{1+\int_{-\infty}^\infty |f'(x)|dx}{1+|k|}.
%\end{align}
By Lemma \ref{estimate for characteristic function}, we know that $\psi_W(t)=O(t^{-2/5})$. Let $E$ be a constant such that $E\geq (\pi/4)^{2/5}$ and $|\psi_W(t)|\leq Et^{-2/5}$ for sufficiently large $t$. Then, for any $\epsilon_0>0$, we have
\begin{align}\label{Fourier inversion in tail regime}
\left|\frac{1}{2\pi}\int_{|t|\geq (E+\epsilon_0)^{5/2}\sqrt{N}}\psi_{Z^{(N)}}(t)e^{-itx}dt\right| &\ \ll \ \int_{t\geq (E+\epsilon_0)^{5/2}\sqrt{N}}\left(\frac{E}{(t/\sqrt{N})^{2/5}}\right)^Ndt \nonumber\\
&\quad \ \ll \ \sqrt{N}E^N\int_{t\geq (E+\epsilon_0)^{5/2}}t^{-2N/5}dt \nonumber\\
&\quad \ \ll \ \frac{\sqrt{N}}{N}\cdot\frac{E^N}{(E+\epsilon_0)^N},
\end{align}
which is of exponential decay in $N$. By \cite{characteristic estimate} and Condition 3, for $|t|\geq \pi/4$, we have
\begin{align}
|\psi_W(t)| \ < \ 1-\frac{c_1}{m^2},
\end{align}
where $m$ is the median of $W$'s which is finite by Condition 3, and $c_1>0$ is some constant. Hence,
\begin{align}\label{Fourier inversion in constant regime}
\left|\frac{1}{2\pi}\int_{\pi\sqrt{N}/4\leq |t|\leq (E+\epsilon_0)^{5/2}\sqrt{N}}\psi_{Z^{(N)}}(t)e^{-itx}dt\right| \ \ll \ \sqrt{N}\left(1-\frac{c_1}{m^2}\right)^N,
\end{align}
which is also of exponential decay in $N$. By \cite{characteristic estimate}, when $0<t<\pi/4$, we have
\begin{align}
|\psi_W(t)| \ < \ 1-\frac{c_2t^2}{m},
\end{align}
where $m$ is again the median of $W_i$'s and $c_2>0$ is some constant. Hence,
\begin{align}\label{Fourier inversion in intermediate regime}
\left|\frac{1}{2\pi}\int_{N^\epsilon\leq |t|\leq \pi\sqrt{N}/4}\psi_{Z^{(N)}}(t)e^{-itx}dt\right| &\ \ll \ \sqrt{N}\left(1-\frac{c_2}{m^2}N^{2\epsilon-1}\right)^N \nonumber\\
&\ \ll \ \sqrt{N}e^{-2c_2N^{2\epsilon}/m^2},
\end{align}
which also decays exponential in $N$. Combining \eqref{Fourier inversion in small regime}, \eqref{Fourier inversion in tail regime}, \eqref{Fourier inversion in constant regime}, \eqref{Fourier inversion in intermediate regime}, we have
\begin{align}
f_{Z^{(N)}}(x) \ = \ \frac{1}{2\pi}\int_{-\infty}^\infty \psi_{Z^{(N)}}(t)e^{-itx}dt \ = \ \varphi(x)+O(N^{-3/2+3\epsilon}),
\end{align}
which is the desired PDF for $Z^{(N)}$. Moreover, by Berry-Esseen theorem \cite{Berry, E}, we have that for all $N$ and $x$,
\begin{align}
F_{Z^{(N)}}(x) \ = \ \Phi(x)+O(N^{-1/2}),
\end{align}
which is the desired estimate for the CDF of $Z^{(N)}$.
\end{proof}
As we shall see for the rest of this section, it suffices for the error term of $f_{Z^{(N)}}$ to be $O(N^{-1/2-\delta})$ and the error term of $F_{Z^{(N)}}$ to be $O(N^{-\delta})$ for some $\delta> 0$. Hence, the estimates of Lemma \ref{PDF and CDF estimates for normalized log side length} are more than enough for this purpose. Now, we go back to \eqref{order statistics joint PDF formula} given in Proposition \ref{order statistics joint PDF proposition}. In our case, the joint PDF of $Z_{(m+1-i)}^{(N)}$ for all $1\leq i\leq d$ is
\begin{align}
&f_{Z^{(N)}_{(m+1-d)},\dots, Z^{(N)}_{(m)}}(z_{m+1-d},\dots, z_{m}) \nonumber\\ 
&\quad \ = \ C_m^{m+1-d}\left(\Phi(z_{m+1-d})+B(z_{m+1-d})\right)^{m-d}\prod_{i=1}^d \left(\varphi(z_{m+i-d})+A(z_{m+i-d})\right),
\end{align}
where $A(x)=O(N^{-1/2-\delta})$ and $B(x)=O(N^{-\delta})$. By binomial expansion,
\begin{align}\label{binomial expansion of PDF power}
\left(\Phi(z_{m+1-d})+B(z_{m+1-d})\right)^{m-d} \ = \ \sum_{j=0}^{m-d}\binom{m-d}{j}(\Phi(z_{m+1-d}))^{m-d-j}(B(z_{m+1-d}))^{j}.
\end{align}
Since $\Phi(x)$ is a CDF, then $0\leq \Phi(x)\leq 1$, a fact which we shall use without explicit mention for the remainder of the paper. Then, any term in \eqref{binomial expansion of PDF power} with $j>1$ is at most of the same order as $B(z_{m+1-d})$. Hence,
\begin{align}
\left(\Phi(z_{m+1-d})+B(z_{m+1-d})\right)^{m-d} \ = \ \left(\Phi(z_{m+1-d})\right)^{m-d}+O(B(z_{m+1-d})).
\end{align}
We make a slight abuse of notation here to write the error term above as $B(z_{m+1-d})$, which does not change the order of the error term. So then we have
\begin{align}\label{joint PDF of order statistics of normalized log side length}
&f_{Z^{(N)}_{(m+1-d)},\dots, Z^{(N)}_{(m)}}(z_{m+1-d},\dots, z_{m}) \nonumber\\ 
&\quad \ = \ C_m^{m+1-d}\left(\Phi(z_{m+1-d})^{m-d}+B(z_{m+1-d})\right)\prod_{i=1}^d \left(\varphi(z_{m+i-d})+A(z_{m+i-d})\right).    
\end{align}

\subsection{PDF of $Y^{(N)}$}\label{section: Probability density function} In this section, we find the PDF of $Y^{(N)}$. The CDF $F_{Y^{(N)}}(y)$ of $Y^{(N)}$ is given by integrating over the appropriate region in $\R^d$, that is, over all the values $(z_{m+1-d},\dots, z_m)$ of $(Z_{(m+1-d)}^{(N)},\dots, Z_{(m)}^{(N)})$ that sums to $y$. We know that $Z_{(m+1-d)}^{(N)},\dots Z_{(m)}^{(N)}$ are not independent, i.e., they have to satisfy $Z_{(m+1-d)}^{(N)}\leq \cdots \leq Z_{(m)}^{(N)}$. In general, we know that any $1\leq j\leq d$, $z_{m+j-d}$ can take value from $z_{m+(j-1)-d}$ to $(y-\sum_{i=1}^{j-1}z_{m+i-d})/(d-j+1)$. The only exception is $z_{m+1-d}$, whose upper bound is already correctly stated but the lower bound is $-\infty$. Hence, we have
\begin{align}
&F_{Y^{(N)}}(y) \nonumber\\
&\quad \ = \ \int_{-\infty}^{\frac{y}{d}}\prod_{j=2}^{d}\int_{z_{m+(j-1)-d}}^{\frac{y-\sum_{i=1}^{j-1}z_{m+i-d}}{d-j+1}} f_{Z^{(N)}_{(m+1-d)},\dots, Z^{(N)}_{(m)}}(z_{m+1-d},\dots, z_m) dz_m\cdots dz_{m+1-d},
\end{align}
where we use the product of integrals to denote
\begin{align}
\prod_{j=k}^{d}\int_{z_j}^{y_j} \ := \ \int_{z_k}^{y_k}\cdots \int_{z_j}^{y_j}\cdots \int_{z_d}^{y_d}.
\end{align}
We separate out the main term $C_m^{m+1-d}\Phi(z_{m+1-d})^{m-d}\prod_{i=1}^d \varphi(z_{m+i-d})$ of the joint PDF \eqref{joint PDF of order statistics of normalized log side length} and denote it by 
\begin{align}\label{main term of joint PDF of order statistics of normalized log side length}
f_{Z^{(N)}_{(m+1-d)}, \dots, Z^{(N)}_{(m)}, \textup{ main}}(z_{m+1-d},\dots, z_m) \ := \ C_m^{m+1-d}\Phi(z_{m+1-d})^{m-d}\prod_{i=1}^d \varphi(z_{m+i-d}).
\end{align}
We also denote the error term of the joint PDF \eqref{joint PDF of order statistics of normalized log side length} by 
$f_{Z^{(N)}_{(m+1-d)}, \dots, Z^{(N)}_{(m)}, \textup{ error}}(z_{m+1-d},\dots, z_m)$. Let $F_{Y^{(N)}, \textup{ main}}(y)$ and $F_{Y^{(N)}, \textup{ error}}(y)$ be the main term and the error term of $F_{Y^{(N)}}(y)$. Then
\begin{align}\label{CDF and main and error terms of normalized max volume}
F_{Y^{(N)}}(y) &\ = \ F_{Y^{(N)}, \textup{ main}}(y)+F_{Y^{(N)}, \textup{ error}}(y) \nonumber\\
F_{Y^{(N)}, \textup{ main}}(y) &\ = \ \int_{-\infty}^{\frac{y}{d}}\prod_{j=2}^{d}\int_{z_{m+(j-1)-d}}^{\frac{y-\sum_{i=1}^{j-1}z_{m+i-d}}{d-j+1}}f_{Z^{(N)}_{(m+1-d)}, \dots, Z^{(N)}_{(m)}, \textup{ main}}(z_{m+1-d},\dots, z_m)dz_m\cdots dz_{m+1-d} \nonumber\\
F_{Y^{(N)}, \textup{ error}}(y) &\ = \ \int_{-\infty}
^{\frac{y}{d}}\prod_{j=2}^d\int_{z_{m+(j-1)-d}}^{\frac{y-\sum_{i=1}^{j-1}z_{m+i-d}}{d-j+1}}f_{Z^{(N)}_{(m+1-d)}, \dots, Z^{(N)}_{(m)}, \textup{ error}}(z_{m+1-d},\dots, z_m)dz_m\cdots dz_{m+1-d}.
\end{align}

We are now ready to find the PDF $f_{Y^{(N)}}$ of $Y^{(N)}$. We know that
\begin{align}
f_{Y^{(N)}}(y) \ = \ \frac{d}{dy}F_{Y^{(N)}}(y) \ = \ \frac{d}{dy}F_{Y^{(N)}, \textup{ main}}(y)+\frac{d}{dy}F_{Y^{(N)}, \textup{ error}}(y).
\end{align}
To differentiate $F_{Y^{(N)}(y), \textup{ main}}$ and $F_{Y^{(N)}, \textup{ error}}(y)$, we apply the \textbf{Leibniz integral rule} for differentiate under the integral sign.
\begin{lemma}\cite{Folland}
Suppose that $f(x,t)$ is a function such that both $f(x,t)$ and $df(x,t)/dx$ are continuous in $t$ and $x$ in the $xt$-plane, and $a(x)$ and $b(x)$ are also continuously differentiable. Then
\begin{align}
\frac{d}{dx}\int_{a(x)}^{b(x)}f(x,t)dt \ = \ f(x,b(x))\cdot \frac{d}{dx}b(x)-f(x,a(x))\cdot \frac{d}{dx}a(x)+\int_{a(x)}^{b(x)}\frac{\partial}{\partial x}f(x,t)dt.
\end{align}
\end{lemma}
For the rest of this paper, one can easily check the regularity conditions on $f(x,t)$, $a(x)$, and $b(x)$ each time we apply the Leibniz integral rule, so there will not be any explicit mention of the regularity conditions again. We first make some observations on the effect of $d/dy$ on

\begin{align}\label{iterated integral generic function}
\int_{-\infty}^{\frac{y}{d}}\prod_{j=2}^d\int_{z_{m+(j-1)-d}}^{\frac{y-\sum_{i=1}^{j-1}z_{m+i-d}}{d-j+1}}g(z_{m+1-d},\dots, z_m)dz_m\cdots dz_{m+1-d}
\end{align}
for any continuous function $g(z_{m+1-d},\dots, z_m)$. One can check that each application of the Leibniz integral rule to \eqref{iterated integral generic function} would only yield one integral, because for the other integral, the lower bound of the integration always coincides with the upper bound of the integration and thus equals to 0. As a result, each application of $d/dy$ here merely interchanges the order of the differentiation and the integral signs. Eventually,
\eqref{iterated integral generic function} becomes
\begin{align}\label{leibniz rule final formula}
&\int_{-\infty}^{\frac{y}{d}}\prod_{j=2}^{d-1}\int_{z_{m+(j-1)-d}}^{\frac{y-\sum_{i=1}^{j-1}z_{m+i-d}}{d-j+1}} \frac{\partial}{\partial y}\left(\int_{z_{m-1}}^{y-\sum_{i=1}^{d-1}z_{m+i-d}}g(z_{m+1-d},\dots, z_m)dz_m\right)dz_{m-1}\cdots dz_{m+1-d} \nonumber\\
&\quad \ = \ \int_{-\infty}^{\frac{y}{d}}\prod_{j=2}^{d-1}\int_{z_{m+(j-1)-d}}^{\frac{y-\sum_{i=1}^{j-1}z_{m+i-d}}{d-j+1}}g\left(z_{m+1-d},\dots, z_{m-1},y-\sum_{i=1}^{d-1}z_{m+i-d}\right)dz_{m-1}\cdots dz_{m+1-d}.
\end{align}
Applying \eqref{leibniz rule final formula} to $dF_{Y^{(N)}}/dy$, we have that the PDF of $Y^{(N)}$ is
\begin{align}\label{PDF of normalized log max volume}
&\frac{d}{dy}F_{Y^{(N)}}(y) \ = \ \int_{-\infty}^{\frac{y}{d}}\prod_{j=2}^{d-1}\int_{z_{m+(j-1)-d}}^{\frac{y-\sum_{i=1}^{j-1}z_{m+i-d}}{d-j+1}}f_{Z^{(N)}_{(m+1-d)}, \dots, Z^{(N)}_{(m)}}\left(z_{m+1-d},\dots, z_{m-1},y-\sum_{i=1}^{d-1}z_{m+i-d}\right) \nonumber\\
&\quad \quad \quad \quad \quad \quad \quad dz_{m-1}\cdots dz_{m+1-d}.
\end{align}
Applying \eqref{leibniz rule final formula} to $dF_{Y^{(N)}, \textup{ main}}(y)/dy$ and $dF_{Y^{(N)}, \textup{ error}}(y)/dy$, we have
\begin{align}\label{main term and error term of PDF of normalized log max volume}
&\frac{d}{dy}F_{Y^{(N)}, \textup{ main}}(y) \nonumber\\
&\quad \ = \ \int_{-\infty}^{\frac{y}{d}}\prod_{j=2}^{d-1}\int_{z_{m+(j-1)-d}}^{\frac{y-\sum_{i=1}^{j-1}z_{m+i-d}}{d-j+1}}f_{Z^{(N)}_{(m+1-d)}, \dots, Z^{(N)}_{(m)}, \textup{ main}}\left(z_{m+1-d},\dots, z_{m-1}, y-\sum_{i=1}^{d-1}z_{m+i-d}\right) \nonumber\\
&\quad\quad\quad dz_{m-1} \cdots dz_{m+1-d} \nonumber\\
&\frac{d}{dy}F_{Y^{(N)}, \textup{ error}}(y) \nonumber\\
&\quad\ = \ \int_{-\infty}^{\frac{y}{d}}\prod_{j=2}^{d-1}\int_{z_{m+(j-1)-d}}^{\frac{y-\sum_{i=1}^{j-1}z_{m+i-d}}{d-j+1}}f_{Z^{(N)}_{(m+1-d)}, \dots, Z^{(N)}_{(m)}, \textup{ error}}\left(z_{m+1-d},\dots, z_{m-1}, y-\sum_{i=1}^{d-1}z_{m+i-d}\right) \nonumber\\
&\quad\quad\quad dz_{m-1} \cdots dz_{m+1-d} \nonumber\\
\end{align}
We would like to make some truncation to the tails of integrals in \eqref{PDF of normalized log max volume} and \eqref{main term and error term of PDF of normalized log max volume}. The truncation method we now describe applies to both $dF_{Y^{(N)}}(y)/dy$ and $dF_{Y^{(N)}, \textup{ main}}(y)/dy$, so we only demonstrate it for $dF_{Y^{(N)}}(y)/dy$. We want to show that
\begin{align}\label{tail of PDF of normalized log max volume}
&\int_{-\infty}^{-N^{\delta'}}\prod_{j=2}^{d-1}\int_{z_{m+(j-1)-d}}^{\frac{y-\sum_{i=1}^{j-1}z_{m+i-d}}{d-j+1}}f_{Z^{(N)}_{(m+1-d)}, \dots, Z^{(N)}_{(m)}}\left(z_{m+1-d},\dots, z_{m-1},y-\sum_{i=1}^{d-1}z_{m+i-d}\right) \nonumber\\
& dz_{m-1}\cdots dz_{m+1-d}
\end{align}
is asymptotically small for any $\delta'$ satisfying $0<\delta'<\delta$. By Proposition \ref{order statistics joint PDF proposition},
\begin{align}
&f_{Z^{(N)}_{(m+1-d)}, \dots, Z^{(N)}_{(m)}}\left(z_{m+1-d},\dots, z_{m-1},y-\sum_{i=1}^{d-1}z_{m+i-d}\right) \nonumber\\ 
&\quad \ = \ C_{m}^{m+1-d} f_{Z^{(N)}}(z_{m+1-d})\cdots f_{Z^{(N)}}(z_{m-1})f_{Z^{(N)}}\left(y-\sum_{i=1}^{d-1}z_{m+i-d}\right)(F_{Z^{(N)}}(z_{m+1-d}))^{m-d}.
\end{align}
By Lemma \ref{PDF and CDF estimates for normalized log side length}, we know that $f_{Z^{(N)}}(z)=\varphi(z)+O(N^{-1/2-\delta})$, and $F_{Z^{(N)}}(z)=\Phi(z)+O(N^{-\delta})$, where $\varphi(z)=1/\sqrt{2\pi}e^{-z^2/2}$ is bounded between $0$ and $1/\sqrt{2\pi}$ and $0\leq \Phi(z)\leq 1$, both of which are $O(1)$. Then,
\begin{align}
f_{Z^{(N)}_{(m+1-d)}, \dots, Z^{(N)}_{(m)}}\left(z_{m+1-d},\dots, z_{m-1},y-\sum_{i=1}^{d-1}z_{m+i-d}\right) \ \ll \ f_{Z^{(N)}}(z_{m+1-d})\cdots f_{Z^{(N)}}(z_{m-1}).
\end{align}
Returning to \eqref{tail of PDF of normalized log max volume}, we have that
\begin{align}
&\int_{-\infty}^{-N^{\delta'}}\prod_{j=2}^{d-1}\int_{z_{m+(j-1)-d}}^{\frac{y-\sum_{i=1}^{j-1}z_{m+i-d}}{d-j+1}}f_{Z^{(N)}_{(m+1-d)}, \dots, Z^{(N)}_{(m)}}\left(z_{m+1-d},\dots, z_{m-1},y-\sum_{i=1}^{d-1}z_{m+i-d}\right) \nonumber\\
& dz_{m-1}\cdots dz_{m+1-d} \nonumber\\
&\quad \ \ll \ \int_{-\infty}^{-N^{\delta'}} \underbrace{\int_{-\infty}^\infty\cdots \int_{-\infty}^\infty}_{d-2} f_{Z^{(N)}}(z_{m+1-d})f_{Z^{(N)}}(z_{m+2-d})\cdots f_{Z^{(N)}}(z_{m-1}) \nonumber\\
&\quad \quad \quad dz_{m-1}\cdots dz_{m+2-d}dz_{m+1-d} \nonumber\\
&\quad \ = \ \int_{-\infty}^{-N^{\delta'}}f_{Z^{(N)}}(z_{m+1-d})dz_{m+1-d}\prod_{i=2}^{d-1}\left(\int_{-\infty}^\infty f_{Z^{(N)}}(z_{m+i-d})dz_{m+i-d}\right).
\end{align}
Since $f_{Z^{(N)}}(z)$ is a PDF, then $\int_{-\infty}^\infty f_{Z^{(N)}}(z)dz=1$. Moreover,
\begin{align}
\int_{-\infty}^{-N^{\delta'}}f_{Z^{(N)}}(z_{m+1-d})dz_{m+1-d} &\ = \ \mathbb{P}\left(Z^{(N)}\leq  -N^{\delta'}\right) \nonumber\\
&\ \leq \ \mathbb{P}\left(|Z^{(N)}-\mathbb{E}[Z^{(N)}]|\geq N^{\delta'}+\mathbb{E}[Z^{(N)}]\right) \nonumber\\
&\ \leq \ \frac{\textup{Var}[Z^{(N)}]}{(N^{\delta'}+\mathbb{E}[Z^{(N)}])^2} \ = \ O(N^{-2\delta'}),
\end{align}
where the last line follows from Chebyshev's inequality and the fact that the mean and the variance of $Z^{(N)}$ are $0$ and $1$ respectively by construction \eqref{log side length expected value and variance} and \eqref{normalized log side length}.
%\begin{align}
%\mathbb{P}\left(\{\exists i\in [m+1-d;m]:|Z_{(i)}|>N^{\delta'}\}\right)
%\end{align}
%is asymptotically small for any $\delta'$ satisfying $0<\delta'<\delta$. We know that
%\begin{align}
%\mathbb{P}\left(\{\exists i\in [m+1-d;m]: |Z_{(i)}|> N^{\delta'}\}\right) 
%&\ \leq \ \sum_{i=m+1-d}^{m}\mathbb{P}\left(|Z_{(i)}|>N^{\delta'}\right).
%\end{align}
%By Chebyshev's inequality, since $Z_{(i)}$'s have bounded mean and variance by construction \eqref{normalized log side length}, then
%\begin{align}
%\mathbb{P}\left(|Z_{(i)}|>N^{\delta'}\right) &\ \leq \ \mathbb{P}\left(|Z_{(i)}-\mathbb{E}[Z_{(i)}]|>N^{\delta'}-\mathbb{E}[Z_{(i)}]\right) \nonumber\\
%&\quad\quad +\mathbb{P}\left(|Z_{(i)}-\mathbb{E}[Z_{(i)}]|>N^{\delta'}+\mathbb{E}[Z_{(i)}]\right) \nonumber\\
%&\ \leq \ \frac{\textup{Var}[Z_{(i)}]}{(N^{\delta'}-\mathbb{E}[Z_{(i)}])^2}+\frac{\textup{Var}[Z_{(i)}]}{(N^{\delta'}+\mathbb{E}[Z_{(i)}])^2} \ = \ O(N^{-2\delta'}).
%\end{align}
Hence, we have established that the tail of $dF_{Y^{(N)}}(y)/dy$ is $O(N^{-2\delta'})$, i.e.,
\begin{align}
&\frac{d}{dy}F_{Y^{(N)}}(y) \nonumber\\
&\quad \ = \ \int_{-N^{\delta'}}^{\frac{y}{d}}\prod_{j=2}^{d-1}\int_{z_{m+(j-1)-d}}^{\frac{y-\sum_{i=1}^{j-1}z_{m+i-d}}{d-j+1}}f_{Z^{(N)}_{(m+1-d)}, \dots, Z^{(N)}_{(m)}}\left(z_{m+1-d},\dots, z_{m-1},y-\sum_{i=1}^{d-1}z_{m+i-d}\right) \nonumber\\
&\quad\quad \quad dz_{m-1}\cdots dz_{m+1-d}+O(N^{-2\delta'}),
\end{align}
By the same truncation method, one can show that the tail of $dF_{Y^{(N)}, \textup{ main}}(y)/dy$ is also $O(N^{-2\delta'})$. The tail of $dF_{Y^{(N)}, \textup{ error}}(y)/dy$ is the difference between the tail of $dF_{Y^{(N)}}(y)/dy$ and that of $dF_{Y^{(N)}, \textup{ main}}(y)/dy$ by definition, then it is also $O(N^{-2\delta'})$. Thus,
\begin{align}\label{truncated main term of PDF of normalized log max volume plus error term}
&\frac{d}{dy}F_{Y^{(N)}, \textup{ main}}(y) \nonumber\\
&\quad \ = \ \int_{-N^{\delta'}}^{\frac{y}{d}}\prod_{j=2}^{d-1}\int_{z_{m+(j-1)-d}}^{\frac{y-\sum_{i=1}^{j-1}z_{m+i-d}}{d-j+1}}f_{Z^{(N)}_{(m+1-d)}, \dots, Z^{(N)}_{(m)}, \textup{ main}}\left(z_{m+1-d},\dots, z_{m-1}, y-\sum_{i=1}^{d-1}z_{m+i-d}\right) \nonumber\\
&\quad\quad\quad dz_{m-1} \cdots dz_{m+1-d}+O(N^{-2\delta'}),
\end{align}
\begin{align}\label{truncated error term of PDF of normalized log max volume plus error term}
&\frac{d}{dy}F_{Y^{(N)}, \textup{ error}}(y) \nonumber\\
&\quad \ = \ \int_{-N^{\delta'}}^{\frac{y}{d}}\prod_{j=2}^{d-1}\int_{z_{m+(j-1)-d}}^{\frac{y-\sum_{i=1}^{j-1}z_{m+i-d}}{d-j+1}}f_{Z^{(N)}_{(m+1-d)}, \dots, Z^{(N)}_{(m)}, \textup{ error}}\left(z_{m+1-d},\dots, z_{m-1}, y-\sum_{i=1}^{d-1}z_{m+i-d}\right) \nonumber\\
&\quad\quad\quad dz_{m-1} \cdots dz_{m+1-d}+O(N^{-2\delta'}).
\end{align}

\subsection{Outline of problem}\label{outline of problem} We are ready to formulate Conjecture \ref{box fragmentation conjecture} in the language that we have set forth in Section \ref{prelim of proof of box fragmentation theorem}. We want to show that $\mathfrak{m}_d^{(N)}$ converges to strong Benford behavior, which is equivalent to showing that $\log_B(\mathfrak{m}_d^{(N)})$ converges to being equidistributed mod 1. We know that $Y^{(N)}:=(\log_B(\mathfrak{m}_d^{(N)})-d\cdot N\mu_P)/\sqrt{N}\sigma_P$. Since equidistribution mod 1 is invariant under translation \cite{Mil1}, then it suffices to show that $\log_B(\mathfrak{m}_d^{(N)})-d\cdot N\mu_P$ converges to being equidistributed mod 1. Moreover, without loss of generality, we can assume $\sigma_P=1$ from now on, as the choice of renormalization would not change our final results. Hence, in terms of the PDF $f_{Y^{(N)}}(y)$, we can formulate our problem as to showing
\begin{align}
F_N(a,b) \ := \ \mathbb{P}\left(Y^{(N)} \textup{ mod 1} \in (a,b)\right) \ = \ \sum_{n=-\infty}^{\infty}\int_{\frac{a+n}{\sqrt{N}}}^{\frac{b+n}{\sqrt{N}}}f_{Y^{(N)}}(y)dy \ \approx \ b-a,
\end{align}
for all $(a,b)\subset (0,1)$ for sufficiently large $N$.
Now, recall that we have truncated the outermost integral of $f_{Y^{(N)}}(y)=dF_{Y^{(N)}}(y)/dy$ at $-N^{\delta'}$ in Section \ref{prelim of proof of box fragmentation theorem}, which is exactly the lower bound for $Z_{(m+1-d)}$. Since $Z_{(m+1-d)}\leq Z_{(m+2-d)}\leq \cdots \leq Z_{(m)}$, then the lower bound for their sum, $Y^{(N)}$, is $-dN^{\delta'}$. Hence, it is reasonable to truncate the sum of $F_N(a,b)$ at $n=\pm dN^{1/2+\delta'}$ and instead show that
\begin{align}
\sum_{n=-dN^{1/2+\delta'}}^{dN^{1/2+\delta'}}\int_{\frac{a+n}{\sqrt{N}}}^{\frac{b+n}{\sqrt{N}}}f_{Y^{(N)}}(y)dy \ \approx \ b-a.
\end{align}
We can make the truncation because the tails of the sum is negligible:
\begin{align}\label{truncation of tail of probability of normalized log max}
\int_{\left|y\right|\geq dN^{\delta'}}f_{Y^{(N)}}(y)dy &\ = \ \mathbb{P}\left(\left|Y^{(N)}\right|\geq dN^{\delta'}\right) \nonumber\\
&\ \leq \ \mathbb{P}\left(\left|Y^{(N)}-\mathbb{E}[Y^{(N)}]\right|\geq dN^{\delta'}-\mathbb{E}[Y^{(N)}]\right) \nonumber\\
&\quad\quad +\mathbb{P}\left(\left|Y^{(N)}-\mathbb{E}[Y^{(N)}]\right|\geq dN^{\delta'}+\mathbb{E}[Y^{(N)}]\right) \nonumber\\
&\ \leq \ \frac{\textup{Var}[Y^{(N)}]}{(dN^{\delta'}-\mathbb{E}[Y^{(N)}])^2}+\frac{\textup{Var}[Y^{(N)}]}{(dN^{\delta'}+\mathbb{E}[Y^{(N)}])^2} \ = \ O(N^{-2\delta'}),
\end{align}
where the last line follows from Chebyshev's inequality and the fact that the mean and the variance of $Y^{(N)}$ are finite, as it is the sum of $d$ random variables whose means and variances are finite. From now on, we shall only consider the bulk part of $F_N(a,b)$. Let the main term and the error term of $F_N(a,b)$ be defined respectively as
\begin{align}\label{main term of probability of normalized log max volume between a and b mod 1}
\mathcal{M}_N(a,b) &\ := \ \sum_{n=-dN^{1/2+\delta'}}^{dN^{1/2+\delta'}}\int_{\frac{a+n}{\sqrt{N}}}^{\frac{b+n}{\sqrt{N}}}\frac{d}{dy}F_{Y^{(N)}, \textup{ main}}(y)dy;
\end{align}
\begin{align}\label{error term of probability of normalized log max volume between a and b mod 1}
\mathcal{E}_N(a,b) &\ := \ \sum_{n=-dN^{1/2+\delta'}}^{dN^{1/2+\delta'}}\int_{\frac{a+n}{\sqrt{N}}}^{\frac{b+n}{\sqrt{N}}}\frac{d}{dy}F_{Y^{(N)}, \textup{ error}}(y)dy.
\end{align}
Our goal is to show
\begin{align}
\mathcal{M}_N(a,b) \ \approx \ b-a, \quad \mathcal{E}_N(a,b) \ \approx \ 0.
\end{align}

\subsection{Upper bound on error term $\mathcal{E}_N(a,b)$}\label{Section: Upper bound on error term} We first want to show that the error term $\mathcal{E}_N(a,b)$ is negligible. We start by giving some bounds on $dF_{Y^{(N)}, \textup{error}}(y)/dy$ from \eqref{main term and error term of PDF of normalized log max volume}. Recall that $f_{Z^{(N)}_{(m+1-d)}, \dots, Z^{(N)}_{(m)}, \textup{ error}}(z_{m+1-d},\dots, z_m)$ is the error term of the joint density function \eqref{joint PDF of order statistics of normalized log side length}. Let $[\ell]$ be short hand for $\{1, \dots, \ell\}$. Then
\begin{align}
&\frac{1}{C_{m}^{m+1-d}}f_{Z^{(N)}_{(m+1-d)}, \dots, Z^{(N)}_{(m)}, \textup{ error}}(z_{m+1-d},\dots, z_m) \nonumber\\
&\quad \ = \ \Phi(z_{m+1-d})^{m-d}\left(\prod_{i=1}^d (\varphi(z_{m+i-d})+A(z_{m+i-d}))-\prod_{i=1}^d\varphi(z_{m+i-d})\right) \nonumber\\
&\quad\quad\quad+B(z_{m+1-d})\prod_{i=1}^d(\varphi(z_{m+i-d})+A(z_{m+i-d})) \nonumber\\
&\quad \ = \ \left(\Phi(z_{m+1-d})^{m-d}+B(z_{m+1-d})\right)\left(\sum_{S\subsetneq [d]}\left(\prod_{i\in S}\varphi(z_{m+i-d})\right)\left(\prod_{i\in [d]\setminus S}A(z_{m+i-d})\right)\right) \nonumber\\
&\quad\quad\quad +B(z_{m+1-d})\prod_{i=1}^d \varphi(z_{m+i-d}) \nonumber\\
&\quad \ \ll \ \sum_{S\subsetneq [d]}\left(\prod_{i\in S}\varphi(z_{m+i-d})\right)\left(\prod_{i\in [d]\setminus S}A(z_{m+i-d})\right)+N^{-\delta}\prod_{i=1}^d \varphi(z_{m+i-d}).
\end{align}
We know from \eqref{main term and error term of PDF of normalized log max volume} that to bound $dF_{Y^{(N)}, \textup{ error}}(y)/dy$, it suffices to bound
\begin{align}\label{first term of error term of PDF of normalized log max volume}
D_1(y) &\ := \ \sum_{S\subsetneq [d]}\int_{-N^{\delta'}}^{\frac{y}{d}}\prod_{j=2}^{d-1}\int_{z_{m+(j-1)-d}}^{\frac{y-\sum_{i=1}^{j-1}z_{m+i-d}}{d-j+1}}\left(\prod_{i\in S}\varphi(z_{m+i-d})\right) \nonumber\\
&\quad\quad \cdot\left(\prod_{i\in [d]\setminus S}A(z_{m+i-d})\right)\Bigg|_{z_m=y-\sum_{i=1}^{d-1}z_{m+i-d}}dz_{m-1}\cdots dz_{m+1-d}
\end{align}
as well as
\begin{align}\label{Second term of error term of PDF of normalized log max volume}
D_2(y) \ := \ N^{-\delta}\int_{-N^{\delta'}}^{\frac{y}{d}}\prod_{j=2}^{d-1}\int_{z_{m+(j-1)-d}}^{\frac{y-\sum_{i=1}^{j-1}z_{m+i-d}}{d-j+1}}\prod_{i=1}^d\varphi(z_{m+i-d})\Bigg|_{z_m=y-\sum_{i=1}^{d-1}z_{m+i-d}}dz_{m-1}\cdots dz_{m+1-d},
\end{align}
because then by \eqref{truncated error term of PDF of normalized log max volume plus error term}, we have
\begin{align}\label{error term of PDF of normalized log max volume decomposed as first term second term and error term}
\frac{d}{dy}F_{Y^{(N)}, \textup{ error}}(y) \ \ll \ D_1(y)+D_2(y)+O(N^{-2\delta'}).
\end{align}

We first look at $D_1(y)$. Suppose that $k_S$ is the largest index such that $k\in [d]\setminus S$. Such an index $k$ exists, since $S\neq [d]$. First, consider $S$'s with $k_S=d$. Then the integrand of \eqref{first term of error term of PDF of normalized log max volume} becomes
\begin{align}
&\left(\prod_{i\in S}\varphi(z_{m+i-d})\right)\left(\prod_{i\in [d]\setminus S}A(z_{m+i-d})\right)\Bigg|_{z_m=y-\sum_{i=1}^{d-1}z_{m+i-d}} \nonumber\\
&\quad \ = \ \left(\prod_{i\in S}\varphi(z_{m+i-d})\right)\left(\prod_{i\in [d-1]\setminus S}A(z_{m+i-d})\right)A\left(y-\sum_{i=1}^{d-1}z_{m+i-d}\right) \nonumber\\
&\quad\ \ll \ N^{-\frac{1}{2}-\delta}\left(\prod_{i\in S}\varphi(z_{m+i-d})\right)\left(\prod_{i\in [d-1]\setminus S}A(z_{m+i-d})\right),
\end{align}
where in the last line we use the fact that $A(x)=O(N^{-1/2-\delta})$. Now, recall the two truncations we performed in \eqref{truncated error term of PDF of normalized log max volume plus error term} and in \eqref{truncation of tail of probability of normalized log max}. These truncations ensure that $z_{m+1-d}$ is at least $-N^{\delta'}$ and that $y$ is bounded above by $dN^{\delta'}$ and below by $-dN^{\delta'}$. Since $y=z_{m+1-d}+\cdots+z_m$ and $z_{m+1-d}\leq \cdots\leq z_m$, then $z_m$ is at most $(2d-1)N^{\delta'}$, which occurs when $y=dN^{\delta'}$ and $z_{m+1-d}=\cdots=z_{m-1}=-N^{\delta'}$. Hence, $y, z_{m+1-d}, \dots, z_{m-1}$ are all $O(N^{\delta'})$, and the lower and upper bounds on the domain of integration for each integral in \eqref{first term of error term of PDF of normalized log max volume} are $O(N^{\delta'})$. Choose a constant $C$ such that $\pm CN^{\delta'}$ are uniform upper and lower bounds on the domains of integration for all the integrals in \eqref{first term of error term of PDF of normalized log max volume}. Hence,
\begin{align}
&\int_{-N^{\delta'}}^{\frac{y}{d}}\prod_{j=2}^{d-1}\int_{z_{m+(j-1)-d}}^{\frac{y-\sum_{i=1}^{j-1}z_{m+i-d}}{d-j+1}}\left(\prod_{i\in S}\varphi(z_{m+i-d})\right)\left(\prod_{i\in [d]\setminus S}A(z_{m+i-d})\right)\Bigg|_{z_m=y-\sum_{i=1}^{d-1}z_{m+i-d}} \nonumber\\
& dz_{m-1}\cdots dz_{m+1-d} \nonumber\\
&\quad\ \ll \ N^{-1/2-\delta}\underbrace{\int_{-CN^{\delta'}}^{CN^{\delta'}}\cdots \int_{-CN^{\delta'}}^{CN^{\delta'}}}_{d-1}\left(\prod_{i\in S}\varphi(z_{m+i-d})\right)\left(\prod_{i\in [d-1]\setminus S}N^{-1/2-\delta}\right)dz_{m-1}\cdots dz_{m+1-d} \nonumber\\
&\quad \ \ll \ N^{-1/2-\delta+|[d-1]\setminus S|(-1/2-\delta+\delta')}\prod_{i\in S}\left(\int_{-CN^{\delta'}}^{CN^{\delta'}}\varphi(z_{m+i-d})dz_{m+i-d}\right) \nonumber\\
&\quad \ \ll \ N^{-1/2-\delta+|[d-1]\setminus S|(-1/2-\delta+\delta')} \ \ll \ N^{-1/2-\delta},
\end{align}
where in the last line we use the fact that $0<\delta'<\delta$, as well as
\begin{align}
\int_{-CN^{\delta'}}^{CN^{\delta'}}\varphi(x)dx \ \ll \ \int_{-\infty}^\infty \varphi(x)dx \ = \ 1,
\end{align}
which we shall use without reference for the remainder of the paper. Since there is only a finite number of subsets $S$ of $[d]$, then \eqref{first term of error term of PDF of normalized log max volume} has the estimate 
\begin{align}\label{first term of error term of PDF of normalized log max volume case I estimate}
D_1(y) \ \ll \ N^{-1/2-\delta}.
\end{align}

We move on to the case when $k_S<d$, which is similar to the $k_S=d$ case. Now, the $O(N^{-1/2-\delta})$ decay that we need comes from the $A(z_{m+k_S-d})$ term. With this in mind, using the same method for the case when $k_S=d$ as well as the convolution formula for Gaussian PDF, we can again show that $D_1(y)\ll N^{-1/2-\delta}$. We leave the details to \ref{appendix kS less than d}.

We now turn to $D_2(y)$ as defined in \eqref{Second term of error term of PDF of normalized log max volume}, which is relatively more straightforward to bound. We have
\begin{align}
D_2(y) \ \ll \ N^{-\delta}\underbrace{\int_{-\infty}^\infty \cdots \int_{-\infty}^\infty}_{d-1}\prod_{i=1}^{d-1}\varphi(z_{m+i-d})\varphi\left(y-\sum_{i=1}^{d-1}z_{m+i-d}\right)dz_{m-1}\cdots dz_{m+1-d}.
\end{align}
We recognize that the integral is the convolution of $d$ standard Gaussian density functions. We know that the convolution is the probability density function of sum of $d$ i.i.d. standard Gaussian random variables, which itself is also with mean 0 and variance $d$. Hence,
\begin{align}\label{second term of error term of PDF of normalized log max volume estimate}
D_2(y) \ \ll \ N^{-\delta}\frac{1}{\sqrt{2\pi d}}e^{-y^2/(2d)}.
\end{align}
Thus, combining the estimates \eqref{first term of error term of PDF of normalized log max volume case I estimate}, \eqref{first term of error term of PDF of normalized log max volume case II estimate} on $D_1(y)$, the estimate \eqref{second term of error term of PDF of normalized log max volume estimate} on $D_2(y)$, as well as \eqref{error term of PDF of normalized log max volume decomposed as first term second term and error term}, we have
\begin{align}
\frac{d}{dy}F_{Y^{(N)}, \textup{ error}}(y) &\ \ll \ D_1(y)+D_2(y)+N^{-2\delta'} \ \ll \ N^{-1/2-\delta}+N^{-\delta}\frac{1}{\sqrt{2\pi d}}e^{-y^2/(2d)}+N^{-2\delta'}.
\end{align}

We are now ready to bound the error term $\mathcal{E}_N(a,b)$ from \eqref{error term of probability of normalized log max volume between a and b mod 1}:
\begin{align}\label{error term of probability of normalized log max volume between a and b mod 1 final estimate}
\mathcal{E}_N(a,b) &\ = \ \sum_{n=-dN^{1/2+\delta'}}^{dN^{1/2+\delta'}}\int_{\frac{a+n}{\sqrt{N}}}^{\frac{b+n}{\sqrt{N}}}\frac{d}{dy}F_{Y^{(N)}, \textup{ error}}(y)dy \nonumber\\
&\ \ll \ N^{-\delta}\int_{-dN^{\delta'}}^{dN^{\delta'}}\frac{1}{\sqrt{2\pi d}}e^{-y^2/(2d)}dy+\int_{-dN^{\delta'}}^{dN^{\delta'}}N^{-1/2-\delta}+N^{-2\delta'}dy
\ \ll \ N^{-\delta'},
\end{align}
where we use the fact that $0<\delta'<\delta<1/2$.

\subsection{Strategy for remainder of proof}
Our next task is to show that the main term  $\mathcal{M}_N(a,b)$ defined in \eqref{main term of probability of normalized log max volume between a and b mod 1} satisfies
\begin{align}
\mathcal{M}_N(a,b) \ := \ \sum_{-dN^{1/2+\delta'}}^{dN^{1/2+\delta'}}\int_{\frac{a+n}{\sqrt{N}}}^{\frac{b+n}{\sqrt{N}}}\frac{d}{dy}F_{Y^{(N)}, \textup{ main}}(y)dy \ \approx \ b-a.
\end{align}
Our strategy is the following. First, we want to write the integral of $dF_{Y^{(N)}, \textup{ main}}(y)/dy$ over each interval $[(a+n)/\sqrt{N},(b+n)/\sqrt{N}]$ as the sum of a constant main term $m_n(a,b)$ together with an error term $e_n(a,b)$ which, as we shall prove, does not accumulate, i.e.,
\begin{align}
\mathcal{M}_{N, \textup{ error}}(a,b) \ = \ \sum_{n=-dN^{1/2+\delta'}}^{dN^{1/2+\delta'}}e_n(a,b) \ \approx \ 0.
\end{align}
Finally, we prove that the sum of the main term
\begin{align}\label{main part of main term of probability of normalized log max volume between a and b mod 1}
\mathcal{M}_{N, \textup{ main}}(a,b) \ = \ \sum_{n=-dN^{1/2+\delta'}}^{dN^{1/2+\delta'}}m_n(a,b)
\end{align}
is a Riemann sum that converges to $b-a$. It then follows that
\begin{align}\label{desired estimate for main term of probability of normalized log max volume between a and b mod 1}
\mathcal{M}_N(a,b) \ = \ \mathcal{M}_{N, \textup{ main}}(a,b)+\mathcal{M}_{N, \textup{ error}}(a,b) \ \approx \ b-a.
\end{align}
Thus, based on the estimate of the error term in \eqref{error term of probability of normalized log max volume between a and b mod 1},
\begin{align}\label{desired estimate for probability of normalized log max volume between a and b mod 1}
F_N(a,b) \ = \ \mathcal{M}_N(a,b)+\mathcal{E}_N(a,b) \ \approx \ b-a,
\end{align}
which establishes the equidistribution result.

\subsection{Equidistribution within small interval}\label{section: equidistribution} In this section, we show that $dF_{Y^{(N)}, \textup{ main}}(y)/dy$ is equidistributed within each small interval $[(a+n)/\sqrt{N},(b+n)/\sqrt{N}]$. For each $n$, we write the integral over $[(a+n)/\sqrt{N},(b+n)/\sqrt{N}]$ as the sum of a constant main term $m_n(a,b)$ and an error term $e_n(a,b)$, i.e.,
\begin{align}
\int_{\frac{a+n}{\sqrt{N}}}^{\frac{b+n}{\sqrt{N}}}\frac{d}{dy}F_{Y^{(N)}, \textup{ main}}(y)dy \ = \ m_n(a,b)+e_n(a,b),
\end{align}
where
\begin{align}\label{main term and error term of probability of normalized log max volume over small interval}
m_n(a,b) &\ := \ \int_{\frac{a+n}{\sqrt{N}}}^{\frac{b+n}{\sqrt{N}}}\frac{d}{dy}F_{Y^{(N)}, \textup{ main}}\left(\frac{n}{\sqrt{N}}\right)dy, \nonumber\\
e_n(a,b) &\ := \ \int_{\frac{a+n}{\sqrt{N}}}^{\frac{b+n}{\sqrt{N}}}\frac{d}{dy}F_{Y^{(N)}, \textup{ main}}(y)-\frac{d}{dy}F_{Y^{(N)}, \textup{ main}}\left(\frac{n}{\sqrt{N}}\right)dy.
\end{align}
Since $dF_{Y^{(N)}, \textup{ main}}(n/\sqrt{N})/dy$ is constant, then the main term $m_n(a,b)$ is
\begin{align}\label{evaluation of main term of probability of normalized log max volume over small interval}
m_n(a,b) \ = \ \frac{b-a}{\sqrt{N}}\frac{d}{dy}F_{Y^{(N)}, \textup{ main}}\left(\frac{n}{\sqrt{N}}\right).
\end{align}

We now quantify the error term. We start by providing an upper bound on
\begin{align}
\left|\frac{d}{dy}F_{Y^{(N)}, \textup{ main}}(y)-\frac{d}{dy}F_{Y^{(N)}, \textup{ main}}\left(\frac{n}{\sqrt{N}}\right)\right|,
\end{align}
for $y\in [n/\sqrt{N},(n+1)/\sqrt{N}]\supset [(a+n)/\sqrt{N},(b+n)/\sqrt{N}]$. Recall from \eqref{truncated main term of PDF of normalized log max volume plus error term} that
\begin{align}
&\frac{d}{dy}F_{Y^{(N)}, \textup{ main}}(y) \nonumber\\
&\quad \ = \ \int_{-N^{\delta'}}^{\frac{y}{d}}\prod_{j=2}^{d-1}\int_{z_{m+(j-1)-d}}^{\frac{y-\sum_{i=1}^{j-1}z_{m+i-d}}{d-j+1}}f_{Z^{(N)}_{(m+1-d)}, \dots, Z^{(N)}_{(m)}, \textup{ main}}\left(z_{m+1-d},\dots, z_{m-1},y-\sum_{i=1}^{d-1}z_{m+i-d}\right) \nonumber\\
&\quad\quad\quad dz_{m-1}\cdots dz_{m+1-d}+O(N^{-2\delta'}) \nonumber\\
&\quad \ = \ \ \int_{-N^{\delta'}}^{\frac{y}{d}}\prod_{j=2}^{d-1}\int_{z_{m+(j-1)-d}}^{\frac{y-\sum_{i=1}^{j-1}z_{m+i-d}}{d-j+1}}\left(\Phi(z_{m+1-d})\right)^{m-d}\prod_{i=1}^{d-1}\varphi(z_{m+i-d}) \nonumber\\
&\quad \quad\quad \cdot\varphi\left(y-\sum_{i=1}^{d-1}z_{m+i-d}\right)dz_{m-1}\cdots dz_{m+1-d}+O(N^{-2\delta'}).
\end{align}
We let
\begin{align}\label{main part of main term of PDF of normalized log max volume}
&I_N(y) \nonumber\\
&\quad \ := \ \int_{-N^{\delta'}}^{\frac{y}{d}}\prod_{j=2}^{d-1}\int_{z_{m+(j-1)-d}}^{\frac{y-\sum_{i=1}^{j-1}z_{m+i-d}}{d-j+1}}f_{Z^{(N)}_{(m+1-d)}, \dots, Z^{(N)}_{(m)}, \textup{ main}}\left(z_{m+1-d},\dots, z_{m-1},y-\sum_{i=1}^{d-1}z_{m+i-d}\right) \nonumber\\
&\quad \quad \quad dz_{m-1}\cdots dz_{m+1-d}.
\end{align}
We see that $I_N(y)$ is a continuous function over $[n/\sqrt{N},(n+1)/\sqrt{N}]$ and a differentiable function over $(n/\sqrt{N},(n+1)/\sqrt{N})$. By the Mean Value Theorem, for any $y_1, y_2\in [n/\sqrt{N}, (n+1)/\sqrt{N}]$ with $y_1<y_2$, we have
\begin{align}
I_N(y_1)-I_N(y_2) \ = \ \frac{d}{dy}I_N(c_n)(y_1-y_2)
\end{align}
for some $c_n\in (y_1, y_2)$. Hence,
\begin{align}\label{MVT applied to main part of main term of PDF of normalized log max volume}
\left|\frac{d}{dy}F_{Y^{(N)}, \textup{ main}}(y_1)-\frac{d}{dy}F_{Y^{(N)}, \textup{ main}}(y_2)\right| &\ \leq \ \left|I_N(y_1)-I_N(y_2)\right|+O(N^{-2\delta'}) \nonumber\\
&\ \leq \ \left|\frac{d}{dy}I_N(c_n)\right||y_1-y_2|+O(N^{-2\delta'})\nonumber\\ 
&\ \ll \ \frac{1}{\sqrt{N}}\left|\frac{d}{dy}I_N(c_n)\right|+O(N^{-2\delta'}).
\end{align}
To obtain an estimate on $dI_N(y)/dy$, we differentiate $I_N(y)$ by following the same procedure as shown in Section \ref{section: Probability density function} to exchange the order of differentiation and integration.
Since the product in the integrand in \eqref{main part of main term of PDF of normalized log max volume} goes from $j=2$ to $d-1$, we need to discuss the cases when $d=2$ and $d>2$ separately. If $d=2$,
\begin{align}
I_N(y) \ = \ \int_{-N^{\delta'}}^{\frac{y}{2}}f_{Z^{(N)}_{(m-1)}, Z^{(N)}_{(m)}, \textup{ main}}\left(z_{m-1},y-z_{m-1}\right)dz_{m-1}.
\end{align}
By the Leibniz integral rule and \eqref{main term of joint PDF of order statistics of normalized log side length},
\begin{align}\label{derivative of main part of main term of PDF of normalized log max volume dimension 2}
&\frac{d}{dy}I_N(y) \nonumber\\
&\quad \ = \ \frac{1}{2}f_{Z^{(N)}_{(m-1)}, Z^{(N)}_{(m)}, \textup{ main}}\left(\frac{y}{2},\frac{y}{2}\right)+\int_{-N^{\delta'}}^{\frac{y}{2}}\frac{\partial }{\partial y}\left(f_{Z^{(N)}_{(m-1)}, Z^{(N)}_{(m)}, \textup{ main}}(z_{m-1},y-z_{m-1})\right)dz_{m-1} \nonumber\\
&\quad \ \ll \ \varphi\left(\frac{y}{2}\right)+\int_{-N^{\delta'}}^{\frac{y}{2}}(\Phi(z_{m-1}))^{m-2}\varphi(z_{m-1})\cdot (-1)\cdot (y-z_{m-1})\varphi(y-z_{m-1})dz_{m-1}.
\end{align}
We know that $y\geq -2N^{\delta'}$, so $y/2\geq -N^{\delta'}$. We now break into cases when $y\leq 0$ and $y\geq 0$. When $y\leq 0$, $\varphi(z_{m-1})$ is at most $\varphi(y/2)$ on $[-N^{\delta'},y/2]$. So \eqref{derivative of main part of main term of PDF of normalized log max volume dimension 2} becomes
\begin{align}
\frac{d}{dy}I_N(y) &\ \ll \ \varphi\left(\frac{y}{2}\right)+\varphi\left(\frac{y}{2}\right)\int_{-\infty}^\infty |y-z_{m-1}|\varphi(y-z_{m-1})dz_{m-1} \nonumber\\
&\ = \ \varphi\left(\frac{y}{2}\right)+\varphi\left(\frac{y}{2}\right)\int_{-\infty}^\infty |z_{m-1}|\varphi(z_{m-1})dz_{m-1} \ \ll \ \varphi\left(\frac{y}{2}\right),
\end{align}
where in the last line we use the fact that the integral $\int_{-\infty}^\infty |x|\varphi(x)dx$ is the expected value of the absolute value of a standard Gaussian random variable, which is finite. When $y\geq 0$, we have
\begin{align}
y-z_{m-1} &\ \geq \ y-\frac{y}{2} \ = \ \frac{y}{2} \ \geq \ 0.
\end{align}
Hence, \eqref{derivative of main part of main term of PDF of normalized log max volume dimension 2} becomes
\begin{align}\label{derivative of main part of main term of PDF of normalized log max volume final estimate dimension 2}
\frac{d}{dy}I_N(y) &\ \ll \ \varphi\left(\frac{y}{2}\right)+\int_{-N^{\delta'}}^{\frac{y}{2}}(y-z_{m-1})\varphi(y-z_{m-1})dz_{m-1} \nonumber\\
&\ = \ \varphi\left(\frac{y}{2}\right)+\varphi(y-z_{m-1})\Bigg|_{z_{m-1}=-N^{\delta'}}^{z_{m-1}=\frac{y}{2}} \ \ll \ \varphi\left(\frac{y}{2}\right)+\varphi\left(y+N^{\delta'}\right),
\end{align}
where in the second line we employ the fact that $\int (y-x)\varphi(y-x)dx=\varphi(y-x)$. Thus, if $d=2$, regardless of whether $y\leq 0$ or $y\geq 0$, $d^2\mathcal{M}_{Y^{(N)}}(y)/dy^2$ is on the order of $\varphi(y/2)+\varphi(y+N^{\delta'})$. This concludes the estimate of $d^2\mathcal{M}_{Y^{(N)}}(y)/dy^2$ for $d=2$.

If $d> 2$, we repeatedly apply Leibniz integral rule as we have demonstrated in Subsection \ref{Section: Upper bound on error term} to \eqref{main part of main term of PDF of normalized log max volume} and also use \eqref{main term of joint PDF of order statistics of normalized log side length} to obtain
\begin{align}\label{derivative of main part of main term of PDF of normalized log max volume high dimension}
&\frac{d}{dy}I_N(y)\nonumber\\
&\ = \ \int_{-N^{\delta'}}^{\frac{y}{d}}\prod_{j=2}^{d-2}\int_{z_{m+(j-1)-d}}^{\frac{y-\sum_{i=1}^{j-1}z_{m+i-d}}{d-j+1}}\frac{\partial}{\partial y}\Bigg(\int_{z_{m-2}}^{\frac{y-\sum_{i=1}^{d-2}z_{m+i-d}}{2}}f_{Z^{(N)}_{(m+1-d)}, \dots, Z^{(N)}_{(m)}, \textup{ main}}\Bigg(z_{m+1-d},\dots, z_{m-1},\nonumber\\
&\quad\quad y-\sum_{i=1}^{d-1}z_{m+i-d}\Bigg)dz_{m-1}\Bigg)dz_{m-2}\cdots dz_{m+1-d} \nonumber\\
&\ \ll \ \int_{-N^{\delta'}}^{\frac{y}{d}}\prod_{j=2}^{d-2}\int_{z_{m+(j-1)-d}}^{\frac{y-\sum_{i=1}^{j-1}z_{m+i-d}}{d-j+1}}\left(\Phi(z_{m+1-d})\right)^{m-d}\left(\prod_{i=1}^{d-2}\varphi(z_{m+i-d})\right)\left(\varphi\left(\frac{y-\sum_{i=1}^{d-2}z_{m+i-d}}{2}\right)\right)^2 \nonumber\\
&\quad\quad + \Bigg(\int_{z_{m-2}}^{\frac{y-\sum_{i=1}^{d-2}z_{m+i-d}}{2}} (\Phi(z_{m+1-d}))^{m-d} \left(\prod_{i=1}^{d-1}\varphi(z_{m+i-d})\right)  (-1) \left(y-\sum_{i=1}^{d-1}z_{m+i-d}\right)\nonumber\\
&\quad\quad\cdot \varphi\left(y-\sum_{i=1}^{d-1}z_{m+i-d}\right) dz_{m-1}\Bigg)dz_{m-2}\cdots dz_{m+1-d}.
\end{align}
Now, similar to the case when $d=2$, we estimate the second integrand in \eqref{derivative of main part of main term of PDF of normalized log max volume high dimension} by breaking down into cases when $(y-\sum_{i=1}^{d-2}z_{m+i-d})/2\leq 0$ and $(y-\sum_{i=1}^{d-2}z_{m+i-d})\geq 0$. We leave the details to \ref{appendix the case for d greater than 2}. Hence, we have $dI_N(y)/dy$ in \eqref{derivative of main part of main term of PDF of normalized log max volume high dimension} is bounded above by
\begin{align}\label{derivative of main part of main term of PDF of normalized log max volume high dimension estimate}
& \underbrace{\int_{-\infty}^{\infty}\cdots \int_{-\infty}^\infty}_{d-2}\left(\prod_{i=1}^{d-2}\varphi(z_{m+i-d})\right)\varphi\left(\frac{y-\sum_{i=1}^{d-2}z_{m+i-d}}{2}\right) dz_{m-2}\cdots dz_{m+1-d} \nonumber\\
&\quad\quad+ \underbrace{\int_{-\infty}^{\infty}\cdots \int_{-\infty}^\infty}_{d-2}\left(\prod_{i=1}^{d-2}\varphi(z_{m+i-d})\right)\varphi\left(y-\left(\sum_{i=1}^{d-2}z_{m+i-d}\right)-z_{m-2}\right)dz_{m-2}\cdots dz_{m+1-d}.
\end{align}
Notice that the first integral in \eqref{derivative of main part of main term of PDF of normalized log max volume high dimension estimate} is the PDF of $W_1+\cdots+W_{d-2}+2W_{d-1}$, where $W_i$'s are i.i.d. standard Gaussians, and the second integral in \eqref{derivative of main part of main term of PDF of normalized log max volume high dimension estimate} is the PDF of $V_1+\cdots+V_{d-2}+3/2V_{d-1}$, where the $V_i$'s are also i.i.d. standard Gaussians. Hence, there exists a constant $C_0>0$ such that
\begin{align}
\frac{d}{dy}I_N(y) \ \ll \ e^{-(C_0y^2)/2}.
\end{align}
Moreover, recall from \eqref{derivative of main part of main term of PDF of normalized log max volume final estimate dimension 2} that when $d=2$,
\begin{align}
\frac{d}{dy}I_N(y) \ \ll \ \varphi\left(\frac{y}{2}\right)+\varphi\left(y+N^{\delta'}\right).
\end{align}
Thus, regardless of whether $d=2$ or $d>2$,
\begin{align}
\frac{d}{dy}I_N(y) \ \ll \ e^{-\frac{D_0y^2}{2}}+e^{-\frac{D_0\left(y+N^{\delta'}\right)^2}{2}},
\end{align}
where $D_0:=\min\{C_0,1/4\}$. Note that $e^{-D_0y^2/2}$ and $\varphi(\sqrt{D_0}(y+N^{\delta'}))$ each has only one global extreme, at $y=0$ and $y=-N^{\delta'}$ respectively. Hence, locally on $[n/\sqrt{N},(n+1)/\sqrt{N}]$ the functions are monotonic and can be bounded above by the sum of its values at the two end points. Hence, for all $n$,
\begin{align}
\frac{d}{dy}I_N(y) \ \ll \ e^{-\frac{D_0\left(\frac{n}{\sqrt{N}}\right)^2}{2}}+e^{-\frac{D_0\left(\frac{n+1}{\sqrt{N}}\right)^2}{2}}+e^{-\frac{D_0\left(\frac{n}{\sqrt{N}}+N^{\delta'}\right)^2}{2}}+e^{-\frac{D_0\left(\frac{n+1}{\sqrt{N}}+N^{\delta'}\right)^2}{2}},
%+\varphi\left(\sqrt{D_0}\left(\frac{n}{\sqrt{N}}+C\sqrt{N}\right)\right)+\varphi\left(\sqrt{D_0}\left(\frac{n+1}{\sqrt{N}}+C\sqrt{N}\right)\right),
\end{align}
for all $y\in [n/\sqrt{N},(n+1)/\sqrt{N}]$. Now, we return to \eqref{MVT applied to main part of main term of PDF of normalized log max volume}. For all $y_1, y_2\in [n/\sqrt{N},(n+1)/\sqrt{N}]$,
\begin{align}
&\left|\frac{d}{dy}F_{Y^{(N)}, \textup{ main}}(y_1)-\frac{d}{dy}F_{Y^{(N)}, \textup{ main}}(y_2)\right| \nonumber\\
&\quad \ \ll \ \frac{1}{\sqrt{N}}\left(e^{-\frac{D_0\left(\frac{n}{\sqrt{N}}\right)^2}{2}}+e^{-\frac{D_0\left(\frac{n+1}{\sqrt{N}}\right)^2}{2}}+e^{-\frac{D_0\left(\frac{n}{\sqrt{N}}+N^{\delta'}\right)^2}{2}}+e^{-\frac{D_0\left(\frac{n+1}{\sqrt{N}}+N^{\delta'}\right)^2}{2}}\right)+O(N^{-2\delta'}).
\end{align}
Thus, the error term $e_n(a,b)$ as defined in \eqref{main term and error term of probability of normalized log max volume over small interval} is bounded by
\begin{align}\label{error term of probability of normalized log max volume over small interval final estimate}
e_n(a,b) &\ \leq \ \int_{\frac{a+n}{\sqrt{N}}}^{\frac{b+n}{\sqrt{N}}}\left|\frac{d}{dy}F_{Y^{(N)}, \textup{ main}}(y)-\frac{d}{dy}F_{Y^{(N)}, \textup{ main}}\left(\frac{n}{\sqrt{N}}\right)\right|dy \nonumber\\
&\ \ll \ \left(\frac{1}{\sqrt{N}}\right)^2\left(e^{-\frac{D_0\left(\frac{n}{\sqrt{N}}\right)^2}{2}}+e^{-\frac{D_0\left(\frac{n+1}{\sqrt{N}}\right)^2}{2}}+e^{-\frac{D_0\left(\frac{n}{\sqrt{N}}+N^{\delta'}\right)^2}{2}}+e^{-\frac{D_0\left(\frac{n+1}{\sqrt{N}}+N^{\delta'}\right)^2}{2}}\right)\nonumber\\
&\quad\quad+O(N^{-1/2-2\delta'}).
\end{align}

\subsection{Upper bound on error term $\mathcal{M}_{N, \textup{ error}}(a,b)$} In this section, we want to prove that the error term of the main term $\mathcal{M}_N(a,b)$ is small, i.e.,
\begin{align}
\mathcal{M}_{N, \textup{ error}}(a,b) \ = \ \sum_{n=-dN^{1/2+\delta'}}^{dN^{1/2+\delta'}}e_n(a,b) \ \approx \ 0.
\end{align}
Based on the estimate of $e_n(a,b)$ in \eqref{error term of probability of normalized log max volume over small interval final estimate}, we can pull out one of the $1/\sqrt{N}$ factors and obtain
\begin{align}
&\mathcal{M}_{N, \textup{ error}}(a,b) \ = \ \sum_{n=-dN^{1/2+\delta'}}^{dN^{1/2+\delta'}}e_n(a,b) \nonumber\\
&\ \ll \ \frac{1}{\sqrt{N}}\cdot \left(\frac{1}{\sqrt{N}}\sum_{n=-dN^{1/2+\delta'}}^{dN^{1/2+\delta'}}\left(e^{-\frac{D_0\left(\frac{n}{\sqrt{N}}\right)^2}{2}}+e^{-\frac{D_0\left(\frac{n+1}{\sqrt{N}}\right)^2}{2}}+e^{-\frac{D_0\left(\frac{n}{\sqrt{N}}+N^{\delta'}\right)^2}{2}}+e^{-\frac{D_0\left(\frac{n+1}{\sqrt{N}}+N^{\delta'}\right)^2}{2}}\right)\right) \nonumber\\
&\quad\quad+O(N^{-\delta'}).
\end{align}
We observe that item in the parentheses above is double of the sum of the Riemann sums for $\int_{-\infty}^\infty e^{-D_0x^2/2}dx$ and $\int_{-\infty}^\infty e^{-D_0(x+N^{\delta'})^2/2}dx$, which are both finite. Hence,
\begin{align}\label{error part of main term of probability of normalized log max volume between a and b mod 1 final estimate}
\mathcal{M}_{N, \textup{ error}}(a,b) \ \ll \ \frac{1}{\sqrt{N}}+N^{-\delta'} \ \ll \ N^{-\delta'}.
\end{align}
We have thus established that the error term $\mathcal{M}_{\textup{err},N}(a,b)$ of the main term $\mathcal{M}_N(a,b)$ is negligible.

\subsection{Evaluation of main term $\mathcal{M}_{N, \textup{ main}}(a,b)$} Finally, we establish the main term $\mathcal{M}_{N, \textup{ main}}(a,b)$ of $\mathcal{M}_{N}(a,b)$. From \eqref{main part of main term of probability of normalized log max volume between a and b mod 1} and \eqref{evaluation of main term of probability of normalized log max volume over small interval}, we see that
\begin{align}
\mathcal{M}_{N, \textup{ main}}(a,b) &\ = \ \sum_{n=-dN^{1/2+\delta'}}^{dN^{1/2+\delta'}}m_n(a,b) \nonumber\\
&\ = \ (b-a)\cdot \left(\frac{1}{\sqrt{N}}\sum_{n=-dN^{1/2+\delta'}}^{dN^{1/2+\delta'}}\frac{d}{dy}F_{Y^{(N)}, \textup{ main}}\left(\frac{n}{\sqrt{N}}\right)\right).
\end{align}
Since the item in the parentheses above is the Riemann sum for $\int_{-\infty}^\infty \frac{d}{dy}F_{Y^{(N)}, \textup{ main}}(y)dy$, we have
\begin{align}
\frac{1}{\sqrt{N}}\sum_{n=-dN^{1/2+\delta'}}^{dN^{1/2+\delta'}}\frac{d}{dy}F_{Y^{(N)}, \textup{ main}}\left(\frac{n}{\sqrt{N}}\right) &\ = \ \int_{-\infty}^\infty \frac{d}{dy}F_{Y^{(N)}, \textup{ main}}(y)dy+o(1) \nonumber\\
&\ = \ F_{Y^{(N)}, \textup{ main}}(\infty)-F_{Y^{(N)}, \textup{ main}}(-\infty)+o(1).
\end{align}
Hence the main term $\mathcal{M}_{N, \textup{ main}}(a,b)$ becomes
\begin{align}\label{main part of main term of probability of normalized log max volume between a and b mod 1 final estimate}
\mathcal{M}_{N, \textup{ main}}(a,b) \ = \ (b-a)(F_{Y^{(N)}, \textup{ main}}(\infty)-F_{Y^{(N), \textup{ main}}}(-\infty))+o(1).
\end{align}
Recall that
\begin{align}
F_{Y^{(N)}, \textup{ main}}(y) &\ = \ \int_{-\infty}^{\frac{y}{d}}\prod_{j=2}^{d}\int_{z_{m+(j-1)-d}}^{\frac{y-\sum_{i=1}^{j-1}z_{m+i-d}}{d-j+1}}f_{Z^{(N)}_{(m+1-d)}, \dots, Z^{(N)}_{(m)}, \textup{ main}}(z_{m+1-d},\dots, z_m)dz_m\cdots dz_{m+1-d}.
\end{align}
Now, by Proposition \ref{order statistics joint PDF proposition} and \eqref{main term of joint PDF of order statistics of normalized log side length},
\begin{align}
f_{Z^{(N)}_{(m+1-d)}, \dots, Z^{(N)}_{(m)}, \textup{ main}}(z_{m+1-d},\dots, z_m) \ := \ C_{m}^{m+1-d}\left(\Phi(z_{m+1-d})\right)^{m-d}\prod_{i=1}^d \varphi(z_{m+i-d})
\end{align}
is the joint PDF of $W_{(m+1-d)}, \dots, W_{(m)}$, where $W_1, \dots, W_m$ are i.i.d. standard Gaussians. Thus, following our derivation in Subsection \ref{section: Probability density function},
\begin{align}
F_{Y^{(N)}, \textup{ main}}(y) \ = \ \int_{-\infty}^{\frac{y}{d}}\prod_{j=2}^{d}\int_{z_{m+(j-1)-d}}^{\frac{y-\sum_{i=1}^{j-1}z_{m+i-d}}{d-j+1}}f_{Z^{(N)}_{(m+1-d)}, \dots, Z^{(N)}_{(m)}, \textup{ main}}(z_{m+1-d},\dots, z_m)dz_m\cdots dz_{m+1-d}
\end{align}
is the CDF of $\sum_{i=1}^d W_{(m+i-d)}$. We know that for a cumulative density function $F(y)$, $F(-\infty)=0$ and $F(\infty)=1$.
Thus, substituting these two estimates into \eqref{main part of main term of probability of normalized log max volume between a and b mod 1 final estimate} yields
\begin{align}
\mathcal{M}_{N, \textup{ main}}(a,b) \ = \ (b-a)+o(1).
\end{align}
Returning to \eqref{desired estimate for probability of normalized log max volume between a and b mod 1} and \eqref{desired estimate for main term of probability of normalized log max volume between a and b mod 1}, combined with the estimate for the error term $\mathcal{E}_N(a,b)$ in \eqref{error term of probability of normalized log max volume between a and b mod 1 final estimate} and the error term $\mathcal{M}_{\textup{err},N}(a,b)$ of $\mathcal{M}_N(a,b)$ in \eqref{error part of main term of probability of normalized log max volume between a and b mod 1 final estimate}, we have
\begin{align}
F_N(a,b) &\ = \ \mathcal{M}_N(a,b)+\mathcal{E}_N(a,b) \nonumber\\
&\ = \ \mathcal{M}_{N, \textup{ main}}(a,b)+\mathcal{M}_{N, \textup{ error}}(a,b)+\mathcal{E}_N(a,b) \nonumber\\
&\ = \ (b-a)+o(1)+O(N^{-\delta'})+O(N^{-\delta'}) \nonumber\\
&\ = \ (b-a)+o(1).
\end{align}
We conclude that $\log_B(\mathfrak{m}_d^{(N)})$ converges to being equidistributed mod 1, and therefore by Uniform Distribution Characterization $\mathfrak{m}_d^{(N)}$ converges to strong Benford behavior.
\qed

\section{Funding}
This work was partially supported by Williams College Summer Science Program Research Fellowship, the Finnerty Fund, and NSF Grant DMS2241623. We would like to thank the anonymous referee for their valuable time and constructive comments.

\appendix
\section{Proof of Proposition \ref{binomial difference bound}}\label{appendix binomial difference bound} 
In this section, we prove Proposition \ref{binomial difference bound}, which provides a quantitative bound on the difference among probabilities within an interval. Let us first recall the statement of Proposition \ref{binomial difference bound}.
\begin{prop: binomial difference bound}
For $\ell\leq \sqrt{k(N)}/2$, 
\begin{align}
\left|\binom{k(N)}{k_{1,\ell}}-\binom{k(N)}{k_{1,\ell+1}}\right| \ \leq \  O\left(\binom{k(N)}{k_{1,\ell}}\cdot N^{-\frac{3\epsilon}{10}}\right).
\end{align}
\end{prop: binomial difference bound}

\begin{proof}
We first factor out $\binom{k(N)}{k_{1,\ell}}$ from the difference:
\begin{align}\label{binomial difference}
&\binom{k(N)}{k_{1,\ell}}-\binom{k(N)}{k_{1,\ell+1}} \ = \ \binom{k(N)}{\frac{k(N)}{2}+\ell N^\delta}-\binom{k(N)}{\frac{k(N)}{2}+(\ell+1)N^\delta} \nonumber\\
&\ = \ \frac{k(N)!}{\left(\frac{k(N)}{2}+\ell N^\delta\right)!\left(\frac{k(N)}{2}-\ell N^\delta\right)!}-\frac{k(N)!}{\left(\frac{k(N)}{2}+(\ell+1)N^\delta\right)!\left(\frac{k(N)}{2}-(\ell+1)N^\delta\right)} \nonumber \\
%&\ = \ \frac{k(N)!\left(\frac{k(N)}{2}+(\ell+1)N^\delta\right)!\left(\frac{k(N)}{2}-(\ell+1)N^\delta\right)!-k(N)!\left(\frac{k(N)}{2}+\ell N^\delta\right)!\left(\frac{k(N)}{2}-\ell N^\delta\right)!}{\left(\frac{k(N)}{2}+\ell N^\delta\right)!\left(\frac{k(N)}{2}-\ell N^\delta\right)!\left(\frac{k(N)}{2}+(\ell+1)N^\delta\right)!\left(\frac{k(N)}{2}-(\ell+1)N^\delta\right)!} \nonumber\\
&\ = \ \frac{k(N)!}{\left(\frac{k(N)}{2}+\ell N^\delta\right)!\left(\frac{k(N)}{2}-\ell N^\delta\right)!} \nonumber\\
&\cdot \frac{\left(\frac{k(N)}{2}+(\ell+1)N^\delta\right)!\left(\frac{k(N)}{2}-(\ell+1)N^\delta\right)!-\left(\frac{k(N)}{2}+\ell N^\delta\right)!\left(\frac{k(N)}{2}-\ell N^\delta\right)!}{\left(\frac{k(N)}{2}+(\ell+1)N^\delta\right)!\left(\frac{k(N)}{2}-(\ell+1)N^\delta\right)!} \nonumber\\
&\ = \ \binom{k(N)}{k_{1,\ell}}\left(1-\frac{\left(\frac{k(N)}{2}+\ell N^\delta\right)!\left(\frac{k(N)}{2}-\ell N^\delta\right)!}{\left(\frac{k(N)}{2}+(\ell+1)N^\delta\right)!\left(\frac{k(N)}{2}-(\ell+1)N^\delta\right)!}\right).
\end{align}
Now, we analyze the term
\begin{align}
\alpha_{\ell,N} \ := \ \frac{\left(\frac{k(N)}{2}+\ell N^\delta\right)!\left(\frac{k(N)}{2}-\ell N^\delta\right)!}{\left(\frac{k(N)}{2}+(\ell+1)N^\delta\right)!\left(\frac{k(N)}{2}-(\ell+1)N^\delta\right)!}.
\end{align} 
We want to show that $\alpha_{\ell,N}\rightarrow 1$ as $N\rightarrow\infty$, so that the difference in \eqref{binomial difference} is asymptotically much smaller than the main term $\binom{k(N)}{k_{1,\ell}}$. We have
\begin{align}\label{upper and lower bound on alpha}
\frac{\left(\frac{k(N)}{2}-(\ell+1)N^\delta\right)^{N^\delta}}{\left(\frac{k(N)}{2}+(\ell+1)N^\delta\right)^{N^\delta}} &\ \leq \ \alpha_{\ell,N} \ \leq \ \frac{\left(\frac{k(N)}{2}-\ell N^\delta\right)^{N^\delta}}{\left(\frac{k(N)}{2}+\ell N^\delta\right)^{N^\delta}} \nonumber\\
\left(1-\frac{4(\ell+1)N^\delta}{k(N)+2(\ell+1)N^\delta}\right)^{N^\delta} &\ \leq \ \alpha_{\ell, N} \ \leq \left(1-\frac{4\ell N^\delta}{k(N)+2\ell N^\delta}\right)^{N^\delta}.
\end{align}
Since $\ell \leq \sqrt{k(N)}/2$, then
\begin{align}
0 \ \leq \ \frac{4\ell N^\delta}{k(N)+2\ell N^\delta} \ \leq \ \frac{4\ell N^\delta}{k(N)} \ \leq \ \frac{2N^{\delta}}{\sqrt{k(N)}}.
\end{align}
Since $k(N)\geq N^\epsilon$, and $\delta\in (0,\epsilon/10)$, then
\begin{align}
0 \ \leq \ \frac{2N^{2\delta}}{\sqrt{k(N)}} &\ \leq \ \frac{2N^{\epsilon/10}}{N^{\epsilon/2}} \ = \ 2N^{-2\epsilon/5} \ = \ O(N^{-2\epsilon/5}).
\end{align}
Similarly, we also have
\begin{align}
\frac{4(\ell+1)N^\delta}{k(N)+2(\ell+1)N^\delta} \ = \ O(N^{-2\epsilon/5}).
\end{align}
Hence, for sufficiently large $N$,
\begin{align}\label{bound between zero and one}
0 & \ < \  1-\frac{4(\ell+1)N^\delta}{k(N)+2(\ell+1)N^\delta} \ < \ 1 \nonumber\\
0 &\ < \ 1-\frac{4\ell N^\delta}{k(N)+2\ell N^\delta} \ < \ 1.
\end{align}
Returning to \eqref{upper and lower bound on alpha}, given \eqref{bound between zero and one}, we have
\begin{align}
\left(1-\frac{4(\ell+1) N^\delta}{k(N)+2(\ell+1)N^\delta}\right)^{N^\delta} \ \leq \ \alpha_{\ell, N} \ \leq \ \left(1-\frac{4\ell N^\delta}{k(N)+2\ell N^\delta}\right)^{N^\delta}.
\end{align}
Using binomial expansion,
\begin{align}
\sum_{j=1}^{N^\delta}\binom{N^\delta}{j}(-1)^j \left(\frac{4(\ell+1)N^\delta}{k(N)+2(\ell+1)N^\delta}\right)^j&\ \leq \ \alpha_{\ell, N}-1 \ \leq  \ \sum_{j=1}^{N^\delta}\binom{N^\delta}{j}(-1)^j\left(\frac{4\ell N^\delta}{k(N)+2\ell N^\delta}\right)^j.
\end{align}
We first bound the right sum in \eqref{upper and lower bound on alpha}. Using the assumption that $\ell \leq \sqrt{k(N)}/2$, we have
\begin{align}
\left|\sum_{j=1}^{N^\delta}\binom{N^\delta}{j}(-1)^j \left(\frac{4\ell N^\delta}{k(N)+2\ell N^\delta}\right)^j\right| &\ \leq \ \sum_{j=1}^{N^\delta}N^{j\delta}\left(\frac{4\ell N^\delta}{k(N)}\right)^j \nonumber\\
&\ \leq \ \sum_{j=1}^{N^\delta}N^{j\delta}\left(\frac{2\sqrt{k(N)}N^\delta}{k(N)}\right)^j \nonumber\\
&\ = \ \sum_{j=1}^{N^\delta}\left(\frac{2N^{2\delta}}{\sqrt{k(N)}}\right)^j \nonumber\\
&\ = \ \frac{2N^{2\delta}}{\sqrt{k(N)}}\cdot \frac{1-\left(\frac{2N^{2\delta}}{\sqrt{k(N)}}\right)^{N^{\delta}}}{1-\frac{2N^{2\delta}}{\sqrt{k(N)}}},
\end{align}
where on the last line we use the geometric series formula. Since $k(N)\geq N^\epsilon$, and $\delta\in (0,\epsilon/10)$, then
\begin{align}
0 \ \leq \ \frac{2N^{2\delta}}{\sqrt{k(N)}} &\ \leq \ \frac{2N^{\epsilon/5}}{N^{\epsilon/2}} \ = \ 2N^{-3\epsilon/10}.
\end{align}
Hence
\begin{align}\label{right sum}
\left|\sum_{j=1}^{N^\delta}\binom{N^\delta}{j}(-1)^j \left(\frac{4\ell N^\delta}{k(N)+2\ell N^\delta}\right)^j\right| &\ \leq \ 2N^{-3\epsilon/10}\cdot \frac{1}{1-2N^{-3\epsilon/10}} \nonumber\\
&\ = \ O\left(N^{-3\epsilon/10}\right).
\end{align}
Similarly, for the left sum in \eqref{upper and lower bound on alpha}, we also have
\begin{align}\label{left sum}
\left|\sum_{j=1}^{N^\delta}\binom{N^\delta}{j}(-1)^j \left(\frac{4(\ell+1) N^\delta}{k(N)+2(\ell+1) N^\delta}\right)^j\right| \ = \ O\left(N^{-3\epsilon/10}\right).
\end{align}
Applying \eqref{right sum} and \eqref{left sum} to \eqref{upper and lower bound on alpha}, we get
\begin{align}\label{alpha final estimate}
\left|\alpha_{\ell,N}-1\right| \ = \ O\left(N^{-3\epsilon/10}\right).
\end{align}
Thus, substituting the estimate \eqref{alpha final estimate} back to \eqref{binomial difference} gives us
\begin{align}
\left|\binom{k(N)}{k_{1,\ell}}-\binom{k(N)}{k_{1,\ell+1}}\right| \ \leq \   O\left(\binom{k(N)}{k_{1,\ell}}\cdot N^{-3\epsilon/10}\right).
\end{align}
\end{proof}

\section{Case for $k_S<d$}\label{appendix kS less than d}
In this appendix, we want to show that $D_1(y)\ll N^{-1/2-\delta}$ when $k_S<d$. Recall that $S$ is a proper subset of $[d]$ and $k_S$ is the largest index such that $k\in [d]\setminus S$, and that $D_1(y)$ is defined in \eqref{first term of error term of PDF of normalized log max volume} to be
\begin{align}\label{reiterate first term of error term of PDF of normalized log max volume}
D_1(y) &\ := \ \sum_{S\subsetneq [d]}\int_{-N^{\delta'}}^{\frac{y}{d}}\prod_{j=2}^{d-1}\int_{z_{m+(j-1)-d}}^{\frac{y-\sum_{i=1}^{j-1}z_{m+i-d}}{d-j+1}}\left(\prod_{i\in S}\varphi(z_{m+i-d})\right) \nonumber\\
&\quad\quad\cdot\left(\prod_{i\in [d]\setminus S}A(z_{m+i-d})\right)\Bigg|_{z_m=y-\sum_{i=1}^{d-1}z_{m+i-d}}dz_{m-1}\cdots dz_{m+1-d}.
\end{align}
Let $[\ell_1;\ell_2]$ denote $\{\ell_1,\dots, \ell_2\}$ if $\ell_1,\ell_2$ are integers such that $\ell_1\leq \ell_2$, and let it be $\emptyset$ otherwise. Then the integrand of \eqref{first term of error term of PDF of normalized log max volume} becomes
\begin{align}
&\left(\prod_{i\in S}\varphi(z_{m+i-d})\right)\left(\prod_{i\in [d]\setminus S}A(z_{m+i-d})\right)\Bigg|_{z_m=y-\sum_{i=1}^{d-1}z_{m+i-d}} \nonumber\\
%&\quad \ = \ \left(\prod_{i\in S\setminus [k_S+1;d]}\varphi(z_{m+i-d})\right)\left(\prod_{i\in [d]\setminus (S\cup \{k_S\})}A(z_{m+i-d})\right) \cdot A(z_{m+k_S-d}) \nonumber\\
%&\quad \quad \cdot \left(\prod_{i\in [k_S+1;d]}\varphi(z_{m+i-d})\right)\Bigg|_{z_m=y-\sum_{i=1}^{d-1}z_{m+i-d}} \nonumber\\
&\quad \ = \ \left(\prod_{i\in S\setminus [k_S+1;d]}\varphi(z_{m+i-d})\right)\left(\prod_{i\in [d]\setminus (S\cup \{k_S\})}A(z_{m+i-d})\right) \cdot A(z_{m+k_S-d}) \nonumber \\
&\quad\quad\quad\cdot \left(\prod_{i\in [k_S+1;d-1]}\varphi(z_{m+i-d})\right)\varphi\left(y-\sum_{i=1}^{d-1}z_{m+i-d}\right).
\end{align}
Hence,
\begin{align}\label{each term in the sum of first term of error term of PDF of normalized log max volume}
&\int_{-N^{\delta'}}^{\frac{y}{d}}\prod_{j=2}^{d-1}\int_{z_{m+(j-1)-d}}^{\frac{y-\sum_{i=1}^{j-1}z_{m+i-d}}{d-j+1}}\left(\prod_{i\in S}\varphi(z_{m+i-d})\right)\left(\prod_{i\in [d]\setminus S}A(z_{m+i-d})\right)\Bigg|_{z_m=y-\sum_{i=1}^{d-1}z_{m+i-d}} \nonumber\\
&dz_{m-1}\cdots dz_{m+1-d} \nonumber\\
&\quad \ = \ \int_{-N^{\delta'}}^{\frac{y}{d}}\prod_{j=2}^{d-1}\int_{z_{m+(j-1)-d}}^{\frac{y-\sum_{i=1}^{j-1}z_{m+i-d}}{d-j+1}}
\left(\prod_{i\in S\setminus [k_S+1;d]}\varphi(z_{m+i-d})\right)\left(\prod_{i\in [d]\setminus (S\cup \{k_S\})}A(z_{m+i-d})\right) \nonumber\\
&\quad\quad\quad \cdot A(z_{m+k_S-d})\cdot \left(\prod_{i=k_S+1}^{d-1}\varphi(z_{m+i-d})\right) \cdot \varphi\left(y-\sum_{i=1}^{d-1}z_{m+i-d}\right)dz_{m-1}\cdots dz_{m+1-d}.
\end{align}
To bound \eqref{each term in the sum of first term of error term of PDF of normalized log max volume}, we first give an estimate on the following
\begin{align}
&\prod_{j=k_S+1}^{d-1}\int_{z_{m+(j-1)-d}}^{\frac{y-\sum_{i=1}^{j-1}z_{m+i-d}}{d-j+1}}\left(\prod_{i=k_S+1}^{d-1}\varphi(z_{m+i-d})\right)\cdot \varphi\left(y-\sum_{i=1}^{d-1}z_{m+i-d}\right)dz_{m-1}\cdots dz_{m+(k_S+1)-d} \nonumber\\
&\quad \ \ll \ \underbrace{\int_{-\infty}^\infty \cdots \int_{-\infty}^\infty}_{d-k_S-1} \left(\prod_{i=k_S+1}^{d-1}\varphi(z_{m+i-d})\right)\cdot \varphi\left(\left(y-\sum_{i=1}^{k_S}z_{m+i-d}\right)-\sum_{i=k_S+1}^{d-1}z_{m+i-d}\right)dz_{m-1} \nonumber\\
&\quad\quad\quad \cdots dz_{m+(k_S+1)-d} \nonumber\\
&\quad \ \ll \ \frac{1}{\sqrt{2\pi(d-k_S)}}e^{-\frac{\left(y-\sum_{i=1}^{k_S}z_{m+i-d}\right)^2}{2(d-k_S)}},
\end{align}
where in the last line we use the fact that the second last line is exactly the convolution of $d-k_S$ standard Gaussian density function evaluated at $y-\sum_{i=1}^{k_S}z_{m+i-d}$, which is exactly the probability density function of sum of $d-k_S$ independent standard Gaussian random variables and thus is itself also Gaussian with mean 0 and variance $d-k_S$, evaluated at $y-\sum_{i=1}^{k_S}z_{m+i-d}$. Hence, \eqref{each term in the sum of first term of error term of PDF of normalized log max volume} is bounded above by
\begin{align}\label{first term of error term of PDF of normalized log max volume case II estimate}
&\int_{-N^{\delta'}}^{\frac{y}{d}}\prod_{j=2}^{k_S-1}\int_{z_{m}+(j-1)-d}^{\frac{y-\sum_{i=1}^{j-1}z_{m+i-d}}{d-j+1}}\left(\prod_{i\in S\setminus [k_S+1;d]}\varphi(z_{m+i-d})\right)\left(\prod_{i\in [d]\setminus (S\cup \{k_S\})}A(z_{m+i-d})\right) \nonumber\\
&\quad \quad \int_{z_m+(k_S-1)-d}^{\frac{y-\sum_{i=1}^{k_S-1}z_{m+i-d}}{d-k_S+1}}A(z_{m+k_S-d})\cdot \frac{1}{\sqrt{2\pi(d-k_S)}}e^{-\frac{\left(y-\sum_{i=1}^{k_S}z_{m+i-d}\right)^2}{2(d-k_S)}}dz_{m+k_S-d}\cdots dz_{m+1-d} \nonumber\\
%&\quad \ \ll \ N^{-\frac{1}{2}-\delta}\int_{-C\sqrt{N}}^{\frac{y}{d}}\prod_{j=2}^{k_S-1}\int_{z_{m}+(j-1)-d}^{\frac{y-\sum_{i=1}^{j-1}z_{m+i-d}}{d-j+1}}\left(\prod_{i\in S\setminus [k_S+1;d]}\varphi(z_{m+i-d})\right)\left(\prod_{i\in [d]\setminus (S\cup \{k_S\})}A(z_{m+i-d})\right) \nonumber\\
%&\quad\quad \quad dz_{m+(k_S-1)-d}\cdots dz_{m+1-d} \nonumber \\
&\quad \ \ll \ N^{-1/2-\delta}\underbrace{\int_{-CN^{\delta'}}^{CN^{\delta'}}\cdots \int_{-CN^{\delta'}}^{CN^{\delta'}}}_{k_S-1}\left(\prod_{i\in S\setminus [k_S+1;d]}\varphi(z_{m+i-d})\right)\left(\prod_{i\in [d]\setminus (S\cup \{k_S\})}A(z_{m+i-d})\right) \nonumber\\
&\quad \quad \quad dz_{m+(k_S-1)-d}\cdots dz_{m+1-d} \nonumber \\
&\quad \ \ll \ N^{-1/2-\delta}\prod_{i\in S\setminus [k_S+1;d]}\left(\int_{-CN^{\delta'}}^{CN^{\delta'}}\varphi(z_{m+i-d})dz_{m+i-d}\right) \nonumber\\
&\quad\quad\quad\cdot \prod_{i\in [d]\setminus (S\cup \{k_S\})}\left(\int_{-CN^{\delta'}}^{CN^{\delta'}}A(z_{m+i-d})dz_{m+i-d}\right) \nonumber\\
&\quad \ \ll \ N^{-1/2-\delta+\left|[d]\setminus (S\cup \{k_S\})\right|(-1/2-\delta+\delta')} \ \ll \ N^{-1/2-\delta},
\end{align}
where in the second line, we use the change of variable $z_{m+k_S-d}\mapsto z_{m+k_S-d}-(\sum_{i=1}^{k_S-1}z_{m+i-d}+y)$ and the fact that
\begin{align}
\int_{-\infty}^\infty \frac{1}{\sqrt{2\pi(d-k_S)}}e^{-\frac{(z_{m+k_S-d})^2}{2(d-k_S)}}dz_{m+k_S-d} \ = \ 1.
\end{align}
Since there is only a finite number of proper subsets $S$ of $[d]$, then when $k_S<d$, we have that $D_1(y)\ll N^{-1/2-\delta}$ by definition of $D_1(y)$ in \eqref{reiterate first term of error term of PDF of normalized log max volume}.

\section{Case for $d>2$}\label{appendix the case for d greater than 2}
In this section, we want to obtain the following estimate on $dI_N(y)/dy$ when $d>2$
\begin{align}
&\frac{d}{dy}I_N(y)\nonumber\\
&\quad \ \ll \  \underbrace{\int_{-\infty}^{\infty}\cdots \int_{-\infty}^\infty}_{d-2}\left(\prod_{i=1}^{d-2}\varphi(z_{m+i-d})\right)\varphi\left(\frac{y-\sum_{i=1}^{d-2}z_{m+i-d}}{2}\right) dz_{m-2}\cdots dz_{m+1-d} \nonumber\\
&\quad\quad\quad+ \underbrace{\int_{-\infty}^{\infty}\cdots \int_{-\infty}^\infty}_{d-2}\left(\prod_{i=1}^{d-2}\varphi(z_{m+i-d})\right)\varphi\left(y-\left(\sum_{i=1}^{d-2}z_{m+i-d}\right)-z_{m-2}\right)dz_{m-2}\cdots dz_{m+1-d}.
\end{align}
First, recall from \eqref{derivative of main part of main term of PDF of normalized log max volume high dimension} that
\begin{align}
&\frac{d}{dy}I_N(y) \nonumber\\
&\quad \ \ll \ \int_{-N^{\delta'}}^{\frac{y}{d}}\prod_{j=2}^{d-2}\int_{z_{m+(j-1)-d}}^{\frac{y-\sum_{i=1}^{j-1}z_{m+i-d}}{d-j+1}}\left(\Phi(z_{m+1-d})\right)^{m-d}\left(\prod_{i=1}^{d-2}\varphi(z_{m+i-d})\right)\left(\varphi\left(\frac{y-\sum_{i=1}^{d-2}z_{m+i-d}}{2}\right)\right)^2 \nonumber\\
&\quad\quad\quad+ \Bigg(\int_{z_{m-2}}^{\frac{y-\sum_{i=1}^{d-2}z_{m+i-d}}{2}} (\Phi(z_{m+1-d}))^{m-d} \left(\prod_{i=1}^{d-1}\varphi(z_{m+i-d})\right)  (-1) \left(y-\sum_{i=1}^{d-1}z_{m+i-d}\right)\nonumber\\
&\quad\quad\quad\cdot \varphi\left(y-\sum_{i=1}^{d-1}z_{m+i-d}\right) dz_{m-1}\Bigg)dz_{m-2}\cdots dz_{m+1-d}.
\end{align}
Since $y$ is an upper bound on $\sum_{i=1}^d z_{m+i-d}$ and $z_{m-2}\leq z_{m-1}\leq z_{m}$, then we have $z_{m-2} \leq z_{m-1}\leq (z_{m-1}+z_{m})/2 \leq (y-\sum_{i=1}^{d-2}z_{m+i-d})/2$. When $(y-\sum_{i=1}^{d-2}z_{m+i-d})/2\leq 0$, $\varphi(z_{m-1})$ is at most $\varphi((y-\sum_{i=1}^{d-2}z_{m+i-d})/2)$. Hence,
\begin{align}
&\int_{z_{m-2}}^{\frac{y-\sum_{i=1}^{d-2}z_{m+i-d}}{2}}(\Phi(z_{m+1-d}))^{m-d}\left(\prod_{i=1}^{d-1}\varphi(z_{m+i-d})\right)(-1)\left(y-\sum_{i=1}^{d-1}z_{m+i-d}\right) \nonumber\\
&\cdot\varphi\left(y-\sum_{i=1}^{d-1}z_{m+i-d}\right)dz_{m-1} \nonumber\\
&\quad \ \ll \ \left(\prod_{i=1}^{d-2}\varphi(z_{m+i-d})\right)\varphi\left(\frac{y-\sum_{i=1}^{d-2}z_{m+i-d}}{2}\right)\int_{-\infty}^\infty \left(y-\sum_{i=1}^{d-1}z_{m+i-d}\right)\nonumber\\
&\quad\quad\quad\cdot \varphi\left(y-\sum_{i=1}^{d-1}z_{m+i-d}\right)dz_{m-1} \nonumber\\
&\quad \ll \ \left(\prod_{i=1}^{d-2}\varphi(z_{m+i-d})\right)\varphi\left(\frac{y-\sum_{i=1}^{d-2}z_{m+i-d}}{2}\right)\int_{-\infty}^\infty |z_{m-1}|\varphi(z_{m-1})dz_{m-1},
\end{align}
where in the last line we use the change of variable $z_{m-1}\mapsto -z_{m-1}+y-\sum_{i=1}^{d-2}z_{m+i-d}$. The integral $\int_{-\infty}^\infty |x|\varphi(x)dx$ is the expected value of the absolute value of a standard Gaussian random variable, which is finite. Hence, when $(y-\sum_{i=1}^{d-2}z_{m+i-d})/2\leq 0$,
\begin{align}\label{estimate I for y leq 0}
&\int_{z_{m-2}}^{\frac{y-\sum_{i=1}^{d-2}z_{m+i-d}}{2}}(\Phi(z_{m+1-d}))^{m-d}\left(\prod_{i=1}^{d-1}\varphi(z_{m+i-d})\right)(-1)\left(y-\sum_{i=1}^{d-1}z_{m+i-d}\right) \nonumber\\
&\cdot \varphi\left(y-\sum_{i=1}^{d-1}z_{m+i-d}\right)dz_{m-1} \nonumber\\
&\quad \ \ll \ \left(\prod_{i=1}^{d-2}\varphi(z_{m+i-d})\right)\varphi\left(\frac{y-\sum_{i=1}^{d-2}z_{m+i-d}}{2}\right).
\end{align}
On the other hand, when $(y-\sum_{i=1}^{d-2}z_{m+i-d})/2\geq 0$,
\begin{align}
y-\sum_{i=1}^{d-1}z_{m+i-d} \ \geq \ y-\sum_{i=1}^{d-2}z_{m+i-d}-\left(\frac{y-\sum_{i=1}^{d-2}z_{m+i-d}}{2}\right) \ = \frac{y-\sum_{i=1}^{d-2}z_{m+i-d}}{2} \ \geq \ 0.
\end{align}
Hence, we find
\begin{align}\label{estimate II for y geq 0}
&\int_{z_{m-2}}^{\frac{y-\sum_{i=1}^{d-2}z_{m+i-d}}{2}}(\Phi(z_{m+1-d}))^{m-d}\left(\prod_{i=1}^{d-1}\varphi(z_{m+i-d})\right)(-1)\left(y-\sum_{i=1}^{d-1}z_{m+i-d}\right) \nonumber\\
&\cdot\varphi\left(y-\sum_{i=1}^{d-1}z_{m+i-d}\right)dz_{m-1} \nonumber\\
&\quad\ \ll \ \left(\prod_{i=1}^{d-2}\varphi(z_{m+i-d})\right)\int_{z_{m-2}}^{\frac{y-\sum_{i=1}^{d-2}z_{m+i-d}}{2}}\left(y-\sum_{i=1}^{d-1}z_{m+i-d}\right)\varphi\left(y-\sum_{i=1}^{d-1}z_{m+i-d}\right)dz_{m-1} \nonumber \\
&\quad\ = \ \left(\prod_{i=1}^{d-2}\varphi(z_{m+i-d})\right)\varphi\left(y-\sum_{i=1}^{d-1}z_{m+i-d}\right)\Bigg|_{z_{m-1}=z_{m-2}}^{z_{m-1}=\frac{y-\sum_{i=1}^{d-2}z_{m+i-d}}{2}} \nonumber\\
&\quad\ = \ \left(\prod_{i=1}^{d-2}\varphi(z_{m+i-d})\right)\left(\varphi\left(\frac{y-\sum_{i=1}^{d-2}z_{m+i-d}}{2}\right)-\varphi\left(y-\left(\sum_{i=1}^{d-2}z_{m+i-d}\right)-z_{m-2}\right)\right).
\end{align}

Thus, if $d>2$, regardless of whether $(y-\sum_{i=1}^{d-2}z_{m+i-d})/2\leq 0$ or $(y-\sum_{i=1}^{d-2}z_{m+i-d})/2\geq 0$, we have by combining \eqref{estimate I for y leq 0} and \eqref{estimate II for y geq 0} that
\begin{align}
&\int_{z_{m-2}}^{\frac{y-\sum_{i=1}^{d-2}z_{m+i-d}}{2}}(\Phi(z_{m+1-d}))^{m-d}\left(\prod_{i=1}^{d-1}\varphi(z_{m+i-d})\right)(-1)\left(y-\sum_{i=1}^{d-1}z_{m+i-d}\right) \nonumber\\
&\quad \cdot \varphi\left(y-\sum_{i=1}^{d-1}z_{m+i-d}\right)dz_{m-1} \nonumber\\
&\quad \ \ll \ \left(\prod_{i=1}^{d-2}\varphi(z_{m+i-d})\right)\left(\varphi\left(\frac{y-\sum_{i=1}^{d-2}z_{m+i-d}}{2}\right)+\varphi\left(y-\left(\sum_{i=1}^{d-2}z_{m+i-d}\right)-z_{m-2}\right)\right). 
\end{align}
Hence, \eqref{derivative of main part of main term of PDF of normalized log max volume high dimension} becomes
\begin{align}
&\frac{d}{dy}I_N(y) \nonumber\\
&\quad \ \ll \ \int_{-N^{\delta'}}^{\frac{y}{d}}\prod_{j=2}^{d-2}\int_{z_{m+(j-1)-d}}^{\frac{y-\sum_{i=1}^{j-1}z_{m+i-d}}{d-j+1}}\left(\prod_{i=1}^{d-2}\varphi(z_{m+i-d})\right)\left(\varphi\left(\frac{y-\sum_{i=1}^{d-2}z_{m+i-d}}{2}\right)\right)^2 \nonumber\\
&\quad\quad\quad + \left(\prod_{i=1}^{d-2}\varphi(z_{m+i-d})\right)\left(\varphi\left(\frac{y-\sum_{i=1}^{d-2}z_{m+i-d}}{2}\right)+\varphi\left(y-\left(\sum_{i=1}^{d-2}z_{m+i-d}\right)-z_{m-2}\right)\right) \nonumber\\
&\quad\quad\quad dz_{m-2}\cdots dz_{m+1-d} \nonumber\\
&\quad \ \ll \  \underbrace{\int_{-\infty}^{\infty}\cdots \int_{-\infty}^\infty}_{d-2}\left(\prod_{i=1}^{d-2}\varphi(z_{m+i-d})\right)\varphi\left(\frac{y-\sum_{i=1}^{d-2}z_{m+i-d}}{2}\right) dz_{m-2}\cdots dz_{m+1-d} \nonumber\\
&\quad\quad+ \underbrace{\int_{-\infty}^{\infty}\cdots \int_{-\infty}^\infty}_{d-2}\left(\prod_{i=1}^{d-2}\varphi(z_{m+i-d})\right)\varphi\left(y-\left(\sum_{i=1}^{d-2}z_{m+i-d}\right)-z_{m-2}\right)dz_{m-2}\cdots dz_{m+1-d},
\end{align}
which is exactly the estimate we need.

\end{document}